\documentclass[12pt,reqno]{amsart}

\usepackage{etex}
\usepackage[utf8]{inputenc}
\usepackage{amsfonts}
\usepackage{amsmath}
\usepackage{amssymb}
\usepackage{amsthm}
\usepackage{mathrsfs}
\usepackage{stmaryrd}
\usepackage{color}
\usepackage[english]{babel}
\usepackage{fontenc}
\usepackage{url}
\usepackage{graphicx}
\usepackage{bookmark}
\usepackage{caption}
\usepackage{epstopdf}
\usepackage{hyphenat}
\usepackage{float}
\usepackage{indentfirst}
\usepackage[export]{adjustbox}
\usepackage{tikz}
\usepackage{color}
\usepackage[all]{xy}
\usetikzlibrary{matrix,arrows,decorations.pathmorphing}
\usepackage{booktabs}
\usepackage{array}
\newcolumntype{P}[1]{>{\centering\arraybackslash}p{#1}}

\allowdisplaybreaks

\newtheorem{theorem}{Theorem}[section]
\newtheorem{lemma}[theorem]{Lemma}
\newtheorem{proposition}[theorem]{Proposition}
\newtheorem{corollary}[theorem]{Corollary}

\theoremstyle{definition}
\newtheorem{remark}[theorem]{Remark}
\newtheorem{example}[theorem]{Example}
\newtheorem{definition}[theorem]{Definition}

\DeclareMathOperator{\Conv}{Conv}
\DeclareMathOperator{\id}{id}

\makeatletter
\def\l@subsection{\@tocline{2}{0pt}{2.5pc}{5pc}{}}
\makeatother

\author{Sandra Di Rocco}
\address{{\small Department of Mathematics, KTH Royal Institute of Technology, SE-100 44 Stockholm, Sweden}}
\email{dirocco@kth.se}

\author{Luca Schaffler}
\address{{\small Dipartimento di Matematica e Fisica, Universit\`a degli Studi Roma Tre, Largo San Leonardo Murialdo 1, 00146, Roma, Italy}}
\email{luca.schaffler@uniroma3.it}


\title[Families of pointed toric varieties and degenerations]{Families of pointed toric varieties and degenerations}

\begin{document}

\maketitle

\begin{abstract}
The Losev--Manin moduli space parametrizes pointed chains of projective lines. In this paper we study a possible generalization to families of pointed degenerate toric varieties. Geometric properties of these families, such as flatness and reducedness of the fibers, are explored via a combinatorial characterization. We show that such families are described by a specific type of polytope fibration which generalizes the twisted Cayley sums, originally introduced to characterize elementary extremal contractions of fiber type associated to projective $\mathbb{Q}$-factorial toric varieties with positive dual defect. The case of a one-dimensional simplex can be viewed as an alternative construction of the permutohedra.
\end{abstract}


\section{Introduction}
In this paper we introduce and study a toric generalization of the Losev--Manin moduli space building upon \cite{AM16} and \cite{ST21}. The Losev--Manin moduli space $\overline{\mathrm{LM}}_{m+2}$ is a projective, smooth, fine moduli space of dimension $m-1$ parametrizing chains of projective lines marked with $m$ smooth `light' points, which are allowed to collide, and two smooth `heavy' points, which cannot collide with any other marked point \cite{LM00}. This is a central example of compact moduli space in algebraic geometry. $\overline{\mathrm{LM}}_{m+2}$ is isomorphic to a specific Hassett's moduli space of weighted stable curves, which can be described as an explicit toric blow up of $\mathbb{P}^{m-1}$ (see \cite[\S\,6.4]{Has03}). Moreover, the moduli space of stable $n$-pointed rational curves $\overline{\mathrm{M}}_{0,m+2}$ is a blow up of $\overline{\mathrm{LM}}_{m+2}$, and this birational relation played a central role in \cite{CT15} to prove that $\overline{\mathrm{M}}_{0,m+2}$ is in general not a Mori Dream Space.

Using tools from the work of Kapranov--Sturmfels--Zelevinski \cite{KSZ91}, a generalization of $\overline{\mathrm{LM}}_{0,m+2}$ to point configurations in $\mathbb{P}^2$ was used in \cite{ST21} to describe a novel geometric and modular compactification of the moduli space of $n$ points in $\mathbb{P}^2$. Motivated by this, in the current paper we consider the general perspective of point configurations in a projective toric variety $X_P$ associated to an arbitrary lattice polytope $P$.


\subsection{Families of interest}

Let $H$ be the dense open subtorus of $X_P$, diagonally embedded in $X_P^m$. The Chow quotient $X_P^m/\!/H$ can be interpreted as a compactification of the moduli space of $m$ light and $k$ heavy points in $X_P$ up to $H$-action, where $k$ is the number of torus fixed points in $X_P$. The normalization of $X_P^m/\!/H$ is isomorphic to the toric variety $X_{\mathcal{Q}(P,m)}$ associated to the \emph{quotient fan} $\mathcal{Q}(P,m)$ (see \S\,\ref{construction-quotient-fan}). By work of Billera--Sturmfels \cite{BS92}, $X_{\mathcal{Q}(P,m)}$ is a projective toric variety associated to the normal fan $\mathcal{Q}(P,m)$ to a polytope $Q(P,m)$ called the \emph{fiber polytope} ($Q(P,m)$ corresponds to the polytope defined by the canonical ample polarization inherited by the Chow quotient $X_P^m/\!/H$). For example, if $P$ is the $1$-dimensional simplex $\Delta_1$, then $Q(\Delta_1,m)$ is the $(m-1)$-dimensional permutohedron and $X_{\mathcal{Q}(\Delta_1,m)}\cong\overline{\mathrm{LM}}_{0,m+2}$. We introduce toric families of pointed, degenerate toric varieties over $X_{\mathcal{Q}(P,m)}$ for an arbitrary lattice polytope $P$. These varieties are governed by projective toric families $f\colon X_{\mathcal{R}(P,m)}\rightarrow X_{\mathcal{Q}(P,m)}$ which resolve the indeterminacies of the rational map $X_P^m\dashrightarrow X_{\mathcal{Q}(P,m)}$ (see \S\,\ref{sec-toric-family-R(P,m)}). The fiber of $f$ over a point in the dense torus of $X_{\mathcal{Q}(P,m)}$ is an appropriate translate by $H$-action of the diagonal $X_P\subseteq X_P^m$. There is also a morphism  $X_{\mathcal{Q}(P,m+1)}\rightarrow X_{\mathcal{Q}(P,m)}$ (see Proposition~\ref{forgetful-maps}), although in general $X_{\mathcal{R}(P,m)}$ is not isomorphic to $X_{\mathcal{Q}(P,m+1)}$ (see Remark~\ref{non-recursive-family-in-general}).


\subsection{Relation with toric stacks}
\label{rel-with-tor-stacks}

In a more general setting, let $V$ be a projective toric variety and $H$ a subtorus of the defining dense torus of $V$. Following \cite{AM16,Mol14}, one can define a toric stack $[V/\!/H]$ (see also \cite{GS15a,GS15b}), which is a Deligne--Mumford stack with underlying coarse moduli space the normalization of $V/\!/H$. The toric stack $[V/\!/H]$ carries a universal family $\mathcal{U}\rightarrow[V/\!/H]$ which is flat with reduced fibers \cite[Proposition~3.4]{AM16}. Ascher and Molcho proved that $[V/\!/H]$ is isomorphic to the Alexeev--Brion moduli stack parametrizing finite torus-equivariant maps from stable toric varieties to $V$ \cite{AB06}, endowed with the logarithmic structure given in \cite[\S\,4.2]{AM16}.

When $V=X_P^m$, the toric variety $X_{\mathcal{R}(P,m)}$ is the coarse moduli space underlying the toric stack $\mathcal{U}$. This is because the fan $\mathcal{R}(P,m)$ is obtained from the data defining $\mathcal{U}$ forgetting the finitely generated and saturated submonoid corresponding to each cone (see \cite[Definition~2.1 and Remark~2.6]{AM16}). In this sense, the current paper analyzes the geometry of the coarse moduli space morphism underlying $\mathcal{U}\rightarrow[X_P^m/\!/H]$.


\subsection{Main results}

The definition of $\mathcal{R}(P,m)$ implies the existence of a toric embedding $X_{\mathcal{R}(P,m)}\hookrightarrow X_P^m\times X_{\mathcal{Q}(P,m)}$ which induces a polarization on $X_{\mathcal{R}(P,m)}$. Let $R(P,m)$ be the lattice polytope with normal fan $\mathcal{R}(P,m)$ corresponding to this polarization. We find that the polytopes $R(P,m)$ can be described as a generalization of the class of \emph{$\pi$-twisted Cayley sums} introduced by Casagrande--Di Rocco in \cite{CDR08}.

\begin{definition}[{\cite[Definition~3.5]{CDR08}}]
\label{def-pi-twisted-Cayley-sum}
Let $\pi\colon M\rightarrow \Lambda$ be a surjective map of lattices and denote by $\pi_\mathbb{R}$ its extension $M_\mathbb{R}=M\otimes\mathbb{R}\rightarrow\Lambda\otimes\mathbb{R}=\Lambda_\mathbb{R}$. Let $R_1,\ldots,R_\ell\subseteq M_\mathbb{R}$ be lattice polytopes which are \emph{normally equivalent} (see Definition~\ref{def-strictly-combinatorially-iso}). Assume that $\pi_\mathbb{R}(R_i)=v_i\in\Lambda$ are distinct and are the vertices of the lattice polytope $F=\Conv(v_1,\ldots,v_\ell)$. Then $R=\Conv(R_1,\ldots,R_\ell)$ is called \emph{$\pi$-twisted Cayley sum} of $R_1,\ldots,R_\ell$.
\end{definition}

If $Y$ is the toric variety associated to the normally equivalent polytopes $R_1,\ldots,R_\ell$ in the definition, then we have a toric morphism $X_R\rightarrow Y$ with fibers isomorphic to $X_F$ \cite[Lemma~3.6]{CDR08}. This contruction describes elementary extremal contractions of fiber type for a projective $\mathbb{Q}$-factorial toric variety $X$ whose dual variety has codimension strictly larger than one.

The toric morphisms $X_R\rightarrow Y$ arising from $\pi$-twisted Cayley sums do not describe $X_{\mathcal{R}(P,m)}\rightarrow X_{\mathcal{Q}(P,m)}$ as the latter may have reducible fibers. By not requiring that the polytopes $R_i$ are only mapped to vertices we obtain an enlarged class of polytopes which we call \emph{generalized} $\pi$-twisted Cayley sums (see Definition~\ref{def-gen-pi-twisted-Cayley-sum} and Figure~\ref{cut3x3x3cube} for an example). We first prove the following result.

\begin{theorem}
\label{morphism-associated-to-twisted-Cayley-sum}
Let $R=R(R_1,\ldots,R_k,\pi)$ be a generalized $\pi$-twisted Cayley sum and let $F=\pi_\mathbb{R}(R)$. Let $Y$ be the toric variety associated to the normally equivalent lattice polytopes $R_i$. Then there exists a toric morphism $X_R\rightarrow Y$ with generic fiber $X_F$.
\end{theorem}

Next, we show that the above structure describes the toric morphisms $X_{\mathcal{R}(P,m)}\rightarrow X_{\mathcal{Q}(P,m)}$.

\begin{theorem}[Theorem~\ref{thm-R(P,m)-is-gen-twisted-Cayley-sum}]
Let $X_P$ be the polarized toric variety associated to a lattice polytope $P$ and let $H\subseteq X_P$ be the dense open torus. Let $\pi\colon M\rightarrow\Lambda$ be the map of character lattices induced by the diagonal embedding $H\hookrightarrow H^m$. The polytope $R(P,m)$ has the structure of generalized $\pi$-twisted Cayley sum with $\pi_\mathbb{R}(R(P,m))=mP$ and $R(P,m)=\Conv(R_1,\ldots,R_k)$ with $R_i$ normally equivalent to $Q(P,m)$.
\end{theorem}

\begin{figure}
\centering
\includegraphics[scale=0.60,valign=t]{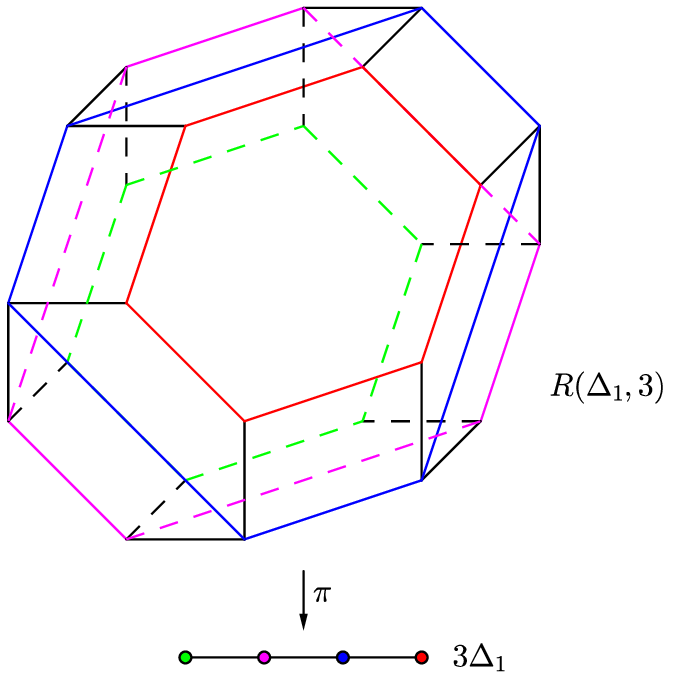}
\caption{The polytope $R(\Delta_1,3)$ as generalized $\pi$-twisted Cayley sum.}
\label{cut3x3x3cube}
\end{figure}

In \S\,\ref{families-on-toric-Chow-quotients} we investigate different geometric features of the family $f\colon X_{\mathcal{R}(P,m)}\rightarrow X_{\mathcal{Q}(P,m)}$. $f$ is equidimensional and it can be endowed with light and heavy sections analogous to the ones of the family of the Losev--Manin moduli space (Proposition~\ref{sections-generalized}). Moreover, we have forgetful morphisms $X_{\mathcal{Q}(P,m+1)}\rightarrow X_{\mathcal{Q}(P,m)}$ (Proposition~\ref{forgetful-maps}). Note that these last two propositions are immediate generalizations of \cite[Lemma~8.6 and Proposition~8.9]{ST21}. More remarks about the structure of the reducible fibers and the relation with the sections can be found in \S\,\ref{on-the-structure-of-reducible-fibers}.

In \S\,\ref{flatness-and-non-reduced-fibers-example} we focus on the following problem. With the same notation as \S\,\ref{rel-with-tor-stacks}, it is known that the fibers of the morphism of coarse moduli spaces corresponding to $\mathcal{U}\rightarrow[V/\!/H]$ are in general not reduced. In the current paper, we give a combinatorial characterization of the reducedness of the fibers of the morphism of coarse moduli spaces when $V=X_P^m$. The starting point is a combinatorial characterization of the cones in $\mathcal{Q}(P,m)$ is terms of certain convex subdivisions $\mathcal{S}(\mathbf{v})$ of $\mathbb{R}^d$ for $\mathbf{v}\in(\mathbb{R}^d)^m$ (see Definition~\ref{def:subdivision-Rd-with-spiders}). We prove the following.

\begin{theorem}[Theorem~\ref{thm-criterion-for-reduced-fibers}]
Consider a toric morphism $f\colon X_{\mathcal{R}(P,m)}\rightarrow X_{\mathcal{Q}(P,m)}$. Then $f$ is flat with reduced fibers if and only if for all vectors $\mathbf{v}\in(\mathbb{Z}^d)^m$ and for all non-empty cells $C\in\mathcal{S}(\mathbf{v})$, there exists a point in $C$ with integral coordinates.
\end{theorem}

As an application of the above criterion, in Theorem~\ref{P-simplex-flat-reduced-fibers} we show that $f$ is flat with reduced fibers if $P$ is the $d$-dimensional simplex $\Delta_d$.

As $X_{\mathcal{R}(P,m)}\rightarrow X_{\mathcal{Q}(P,m)}$ together with the light and heavy sections is a generalization of the Losev--Manin moduli space $\overline{\mathrm{LM}}_{m+2}$, it is natural to ask whether $X_{\mathcal{R}(\Delta_1,m)}$ coincides with the universal family $\overline{\mathrm{LM}}_{m+3}$ over $\overline{\mathrm{LM}}_{m+2}$. In other words, one can ask if the fan $\mathcal{R}(\Delta_1,m)$ is isomorphic to $\mathcal{Q}(\Delta_1,m+1)$. In section \S \ref{family-losev-manin-space-revisited}  we prove that this is indeed the case.
\begin{theorem}[Theorem~\ref{thm-losev-manin-family-revisited}]
The fans $\mathcal{R}(\Delta_1,m)$ and $\mathcal{Q}(\Delta_1,m+1)$ are isomorphic.
\end{theorem}

This implies that $R(\Delta_1,m)$ is the $m$-dimensional permutohedron. Note that the recursion $\mathcal{R}(P,m)\cong\mathcal{Q}(P,m+1)$ is not true in general: we give a counterexample in Remark~\ref{non-recursive-family-in-general}. As $R(P,1)=P$ for any lattice polytope $P$ and $R(\Delta_1,m)$ are the permutohedra, the next new example of polytope $R(P,m)$ to study is the $4$-dimensional polytope $R(\Delta_2,2)$. In \S\,\ref{explicit-examples} we describe $R(\Delta_2,2)$ by listing the maximal dimensional cones of its normal fan. We note that $R(\Delta_2,2)$ is not a simple polytope, showing that in general the toric varieties $X_{R(P,m)}$ are not $\mathbb{Q}$-factorial.


\subsection*{Acknowledgements}
We would like to thank Samouil Molcho, Jenia Tevelev, and Lorenzo Venturello for useful conversations. We also thank the anonymous referee for the valuable comments and feedback. The first author was supported by VR grant [NT:2018-03688]. The second author was supported by the KTH grant Verg Foundation.


\section{Generalized \texorpdfstring{$\pi$}{Lg}-twisted Cayley sums}

\begin{definition}
\label{def-strictly-combinatorially-iso}
Let $P$ and $Q$ be two polytopes of the same dimension and denote by $\mathcal{B}(P)$ and $\mathcal{B}(Q)$ their respective collections of faces. Then $P$ and $Q$ are called \emph{strictly combinatorially isomorphic} if there exists a bijective, inclusion-preserving map $\varphi\colon\mathcal{B}(P)\rightarrow\mathcal{B}(Q)$ such that for any face $F\in\mathcal{B}(P)$, the affine hulls of $F$ and $\varphi(F)$ are translates of each other \cite[Definition~2.11]{Ewa96}. Under this assumption, we further say that $P$ and $Q$ are \emph{normally equivalent} if their (inner) normal fans coincide. Notice that this assumption implies that $P$ and $Q$ define the same toric variety.
\end{definition}

\begin{definition}
\label{def-gen-pi-twisted-Cayley-sum}
Let $\pi\colon M\rightarrow \Lambda$ be a surjective map of lattices and denote by $\pi_\mathbb{R}$ its extension $M_\mathbb{R}=M\otimes\mathbb{R}\rightarrow\Lambda\otimes\mathbb{R}=\Lambda_\mathbb{R}$. Let $R_1,\ldots,R_k\subseteq M_\mathbb{R}$ be normally equivalent lattice polytopes and let $R=\Conv(R_1,\ldots,R_k)$. Assume that $\pi_\mathbb{R}(R_i)=v_i$ is a lattice point in $\Lambda$ for all $i$ and let $F=\Conv(v_1,\ldots,v_k)$. Then $\pi_\mathbb{R}|_R\colon R\rightarrow F$ is a polytope fibration which we call \emph{generalized $\pi$-twisted Cayley sum}. For simplicity, we sometimes write $R=R(R_1,\ldots,R_k,\pi)$.
\end{definition}

\begin{example}
\label{example-generaliz-Cayley-sum}
Let $M=\mathbb{Z}^2$, $\Lambda=\mathbb{Z}$, and define $\pi\colon M\rightarrow\Lambda$ such that $(x,y)\mapsto x+y$. The following lattice polytopes in $M_\mathbb{R}$ are normally equivalent:
\[
R_1=\Conv\{(-2,-1),(-1,-2)\},~R_2=\Conv\{(-2,2),(2,-2)\},~R_3=\Conv\{(1,2),(2,1)\}.
\]
Note that $\pi(R_1)=-3$, $\pi(R_2)=0$, $\pi(R_3)=3$, so let $F\subseteq\Lambda_\mathbb{R}$ be the segment $[-3,3]$. Then $\pi_\mathbb{R}\colon R=\Conv(R_1,R_2,R_3)\rightarrow F$ is a generalized $\pi$-twisted Cayley sum. This is pictured in Figure~\ref{example-generalized-twisted-Cayley-sum}.
\end{example}

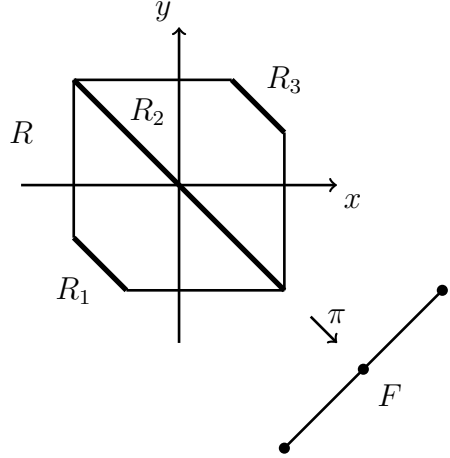
\begin{figure}
\begin{tikzpicture}[scale=0.7]

	\draw[line width=1pt] (-2,2) -- (1,2);
	\draw[line width=1pt] (-2,2) -- (-2,-1);
	\draw[line width=2pt] (-2,-1) -- (-1,-2);
	\draw[line width=1pt] (2,-2) -- (-1,-2);
	\draw[line width=1pt] (2,-2) -- (2,1);
	\draw[line width=2pt] (2,1) -- (1,2);
	\draw[line width=2pt] (-2,2) -- (2,-2);

	\draw[->,line width=1pt] (2.5,-2.5) -- (3,-3);

	\draw[line width=1pt] (-1.5+3.5,-1.5-3.5) -- (1.5+3.5,1.5-3.5);
	\fill (-1.5+3.5,-1.5-3.5) circle (3pt);
	\fill (1.5+3.5,1.5-3.5) circle (3pt);
	\fill (3.5,-3.5) circle (3pt);

	\draw[->,line width=1pt] (-3,0) -- (3,0);
	\draw[->,line width=1pt] (0,-3) -- (0,3);

	\node at (3.3,-0.3) {$x$};
	\node at (-0.3,3.3) {$y$};

	\node at (3,-2.5) {$\pi$};

	\node at (-2,-2) {$R_1$};
	\node at (-0.6,1.4) {$R_2$};
	\node at (2,2) {$R_3$};
	\node at (-3,1) {$R$};

	\node at (4,-4) {$F$};

\end{tikzpicture}
\caption{Generalized $\pi$-twisted Cayley sum in Example~\ref{example-generaliz-Cayley-sum}.}
\label{example-generalized-twisted-Cayley-sum}
\end{figure}

We now prove Theorem~\ref{morphism-associated-to-twisted-Cayley-sum}, which illustrates how the generalized $\pi$-twisted Cayley sums in Definition~\ref{def-gen-pi-twisted-Cayley-sum} give rise to toric morphisms. We restate the theorem here for the reader's convenience.

\begin{theorem}
Let $R=R(R_1,\ldots,R_k,\pi)$ be a generalized $\pi$-twisted Cayley sum and let $F=\pi_\mathbb{R}(R)$. Let $Y$ be the toric variety associated to the normally equivalent lattice polytopes $R_i$. Then there exists a toric morphism $X_R\rightarrow Y$ with generic fiber $X_F$.
\end{theorem}

\begin{proof}
We start by establishing some notation. Given a polytope $P$ and a vertex $v$ of $P$, denote by $\sigma_v^P$ the cone normal to $P$ at $v$. Let $\Delta\rightarrow N$ be the inclusion of lattices dual to $\pi\colon M\rightarrow\Lambda$ and let $q\colon N\rightarrow N/\Delta$ be the quotient map. The fan $\Sigma_R$ normal to $R$ is contained in $N_\mathbb{R}$ and the fan $\Sigma_Y$ is contained in $(N/\Delta)_\mathbb{R}$. Let $\{v_{i_1},\ldots,v_{i_\ell}\}$ be the vertices of $F$ and let $R'=\Conv(R_{i_1},\ldots,R_{i_\ell})$. Then $R'$ is a $\pi$-twisted Cayley sum as in Definition~\ref{def-pi-twisted-Cayley-sum}, and its normal fan is contained in $N_\mathbb{R}$. By \cite[Lemma~3.6]{CDR08} we have a toric morphism $X_{R'}\rightarrow Y$. To guarantee the existence of a toric morphism $X_R\rightarrow Y$ we show that for each vertex $w$ of $R$, there exists a cone $\sigma\in\Sigma_Y$ such that $q_\mathbb{R}(\sigma_w^R)\subseteq\sigma$.

First let $w$ be a vertex of $R$ which maps to a vertex of $F$. In this case, $w$ is also a vertex of $R'$, and since $R'\subseteq R$ we have that $\sigma_w^R\subseteq\sigma_w^{R'}$. As we have a toric morphism $X_{R'}\rightarrow Y$, there exists $\sigma\in\Sigma_Y$ such that $q_\mathbb{R}(\sigma_w^{R'})\subseteq\sigma$, hence $q_\mathbb{R}(\sigma_w^{R})\subseteq\sigma$. Therefore, let us focus on the case where $w$ is a vertex of $R$ not mapping to a vertex of $F$.

As $w$ is a vertex of $R_i$ for some $i$, let $w_{i_1},\ldots,w_{i_\ell}$ be the vertices of $R_{i_1},\ldots,R_{i_\ell}$ corresponding to $w$ under the normally equivalent assumption. It follows that $w_{i_1},\ldots,w_{i_\ell}$ are vertices of $R'$, and for simplicity of notation denote $\sigma_{w_{i_j}}^{R'}$ by $\sigma_{i_j}^{R'}$. By the proof of \cite[Lemma~3.6]{CDR08}, these cones can be written as
\[
\sigma_{i_j}^{R'}=\tau_{i_j}'+\eta',
\]
where the cones $\tau_{i_j}'\subseteq\Delta_\mathbb{R}$ correspond to the faces $R_{i_j}$ of $R'$ and the cone $\eta'$, which satisfies $\eta'\cap\Delta_\mathbb{R}=\{0\}$, is the cone corresponding to the face $Q$ of $R'$ with vertices $w_{i_j}$. Let $R''$ be the convex hull of $R'$ and $w$. Note that $\sigma_w^R\subseteq\sigma_w^{R''}$ because $R''\subseteq R$. Also, any extremal ray of $\sigma_w^{R''}$ is an extremal ray of $\sigma_{i_j}^{R''}$ for some $j$. Since $\sigma_{i_j}^{R''}\subseteq\sigma_{i_j}^{R'}$ the extremal rays of $\sigma_w^{R''}$ are contained in $\cup_j\sigma_{i_j}^{R'}$, hence
\[
\sigma_w^R\subseteq\sigma_w^{R''}\subseteq\Conv(\cup_j\sigma_{i_j}^{R'})=\Conv(\cup_j(\tau_{i_j}'+\eta'))\subseteq\Delta_\mathbb{R}+\eta'.
\]
This implies that $q_\mathbb{R}(\sigma_w^R)\subseteq q_\mathbb{R}(\Delta_\mathbb{R}+\eta')=q_\mathbb{R}(\eta')$. Now let $j\in\{1,\ldots,\ell\}$ be arbitrary. As $X_{R'}\rightarrow Y$ is a toric morphism there exists $\sigma\in\Sigma_Y$ such that $q_\mathbb{R}(\sigma_{i_j}^{R'})\subseteq\sigma$. But $q_\mathbb{R}(\sigma_{i_j}^{R'})=q_\mathbb{R}(\eta')$, hence $q_\mathbb{R}(\sigma_w^R)\subseteq q_\mathbb{R}(\eta')\subseteq\sigma$. Therefore, we have a toric morphism $X_R\rightarrow Y$.
\end{proof}

\begin{remark}
The morphism $X_R\rightarrow Y$ arising from the generalized $\pi$-twisted Cayley sum $\pi\colon R\rightarrow F$ in Example~\ref{example-generaliz-Cayley-sum} has reducible fibers. To see this let $\Sigma_R$ be the fan in $N_\mathbb{R}\cong\mathbb{R}^2$ generated by the following cones:
\begin{gather*}
\sigma_1=\langle(1,0),(1,1)\rangle,~\sigma_2=\langle(1,1),(0,1)\rangle,~\sigma_3=\langle(0,1),(-1,0)\rangle,\\
\sigma_4=\langle(-1,0),(-1,-1)\rangle,~\sigma_5=\langle(-1,-1),(0,-1)\rangle,~\sigma_6=\langle(0,-1),(1,0)\rangle.
\end{gather*}
After identifying $N_\mathbb{R}/\Delta$ with $\mathbb{R}$ via $[(x,y)]\mapsto x-y$, we have that $\Sigma_Y=\{\mathbb{R}_{\geq0},\mathbb{R}_{\leq0},\{0\}\}$. Then
\[
q(\sigma_1)=q(\sigma_5)=q(\sigma_6)=\mathbb{R}_{\geq0},~q(\sigma_2)=q(\sigma_3)=q(\sigma_4)=\mathbb{R}_{\leq0}.
\]
From this we see that the fibers of $X_R\rightarrow Y$ over the torus fixed points corresponding to $\mathbb{R}_{\geq0}$ and $\mathbb{R}_{\leq0}$ are isomorphic to the gluing of two copies of $\mathbb{P}^1$ at a point.
\end{remark}


\section{Families over toric chow quotients}
\label{families-on-toric-Chow-quotients}


\subsection{The Chow quotient $X_P^m/\!/H$, the quotient fan, and the fiber polytope}
\label{construction-quotient-fan}

Let $V$ be a normal projective toric variety and let $H$ be a subtorus of its defining dense open torus $T$. The Chow quotient $V/\!/H$ is a projective compactification of the torus $T/H$ obtained as follows. Given $x\in T$, denote by $Z_x$ the cycle $\overline{H\cdot x}$ with the reduced scheme structure, where the closure is taken inside $V$. Then $V/\!/H$ is the Zariski closure in the Chow variety of $V$ of the subset of points corresponding to the cycles $Z_x$ for $x\in T$. By \cite{KSZ91}, the normalization of the Chow quotient $V/\!/H$ is also a toric variety, which is associated to the so called \emph{quotient fan}. We review the definition of such fan in the case $V=X_P^m$, where $X_P$ is the projective toric variety associated to a full dimensional lattice polytope $P\subseteq\mathbb{R}^d$, and $H\subseteq X_P$ is the dense open torus which we view diagonally embedded in $X_P^m$. The Chow quotient $X_P^m/\!/H$ is a compactification of the torus $H^m/H$, which parametrizes, up to $H$-action, configurations of $m$ points in $H$ (these can coincide).

Let $(\mathbb{Z}^d)^m$ be the lattice of one-parameter subgroups of $H^m$ and let $\mathbb{Z}^d\subseteq(\mathbb{Z}^d)^m$ be the diagonal, which corresponds to the diagonally embedded $H\subseteq H^m$. Tensoring with $\mathbb{R}$ gives the short exact sequence
\begin{equation}
\label{s-e-s-N-lattices-tensor-R}
0\rightarrow\mathbb{R}^d\xrightarrow{\;\delta\;}(\mathbb{R}^d)^m\xrightarrow{\;q\;}(\mathbb{R}^d)^m/\mathbb{R}^d\rightarrow0.
\end{equation}
Let $\Sigma_P$ be the normal fan to $P$. Then the quotient fan, which we denote by $\mathcal{Q}(P,m)$, is the common refinement of the images of the cones in $\Sigma_P^m$ through the quotient map $q$. By \cite[Theorem~2.1~(a)]{KSZ91} we know that
\[
X_{\mathcal{Q}(P,m)}\cong(X_P^m/\!/H)^\nu,
\]
where, for a scheme $X$, $X^\nu$ denotes the normalization of $X$ with the reduced scheme structure.

An important feature of the quotient fan $\mathcal{Q}(P,m)$ is due to Billera and Sturmfels \cite{BS92}: $\mathcal{Q}(P,m)$ is the normal fan to the \emph{fiber polytope} of the polytope fibration $\delta^*|_{P^m}\colon P^m\rightarrow mP$, where $\delta^*\colon(\mathbb{R}^d)^{m*}\rightarrow\mathbb{R}^{d*}$ is the surjection dual to the diagonal embedding $\delta\colon\mathbb{R}^d\hookrightarrow(\mathbb{R}^d)^m$ in \eqref{s-e-s-N-lattices-tensor-R}. Note that, under the natural identification $(\mathbb{R}^d)^{m*}\cong(\mathbb{R}^{d*})^m$, we have $\delta^*(f_1,\ldots,f_m)=f_1+\ldots+f_m$.

\begin{definition}
\label{def:polytope-Q(P,m)}
We denote by $Q(P,m)$ the fiber polytope normal to the quotient fan $\mathcal{Q}(P,m)$ and corresponding to the polarization $\mathscr{M}$ on $X_{\mathcal{Q}(P,m)}$ induced by the projective Chow quotient $X_P^m/\!/H$ (for the polarization on $X_P^m/\!/H$ we refer to \cite[\S\,0.1]{Kap93}). Additionally, we fix a linearization of $\mathscr{M}$, so that we can view $Q(P,m)\subseteq(\mathbb{R}^d)^m/\mathbb{R}^d$.
\end{definition}


\subsection{A toric family over $X_{\mathcal{Q}(P,m)}$}
\label{sec-toric-family-R(P,m)}

It is possible to construct a toric family over $X_{\mathcal{Q}(P,m)}$ that resolves the indeterminacies of the the rational map $X_P^m\dashrightarrow X_{\mathcal{Q}(P,m)}$. Consider the injection
\begin{align*}
\iota=\id\times q\colon(\mathbb{R}^d)^m&\hookrightarrow(\mathbb{R}^d)^m\times((\mathbb{R}^d)^m/\mathbb{R}^d),\\
\mathbf{v}&\mapsto(\mathbf{v},[\mathbf{v}]).
\end{align*}
Then define the fan $\mathcal{R}(P,m)$ as the restriction to $(\mathbb{R}^d)^m$ under $\iota$ of $\Sigma_P^m\times\mathcal{Q}(P,m)$:
\[
\mathcal{R}(P,m)=\iota^{-1}(\Sigma_P^m\times\mathcal{Q}(P,m))=\{q^{-1}(\xi)\cap\tau\mid\xi\in\mathcal{Q}(P,m)~\textrm{and}~\tau\in\Sigma_P^m\}.
\]
As a consequence of this definition, the toric variety $X_{\mathcal{R}(P,m)}$ fits in the following commutative diagram:
\begin{center}
\begin{tikzpicture}[>=angle 90]
\matrix(a)[matrix of math nodes,
row sep=2em, column sep=2em,
text height=1.5ex, text depth=0.25ex]
{X_{\mathcal{R}(P,m)}&X_P^m\times X_{\mathcal{Q}(P,m)}\\
X_P^m&X_{\mathcal{Q}(P,m)}.\\};
\path[right hook->] (a-1-1) edge node[]{}(a-1-2);
\path[->] (a-1-1) edge node[]{}(a-2-1);
\path[dashed,->] (a-2-1) edge node[below]{}(a-2-2);
\path[->] (a-1-2) edge node[]{}(a-2-2);
\end{tikzpicture}
\end{center}
As $X_{\mathcal{R}(P,m)}\hookrightarrow X_P^m\times X_{\mathcal{Q}(P,m)}$ is an embedding, $X_{\mathcal{R}(P,m)}$ is projective, so we can contruct a lattice polytope with normal fan $\mathcal{R}(P,m)$.

\begin{definition}
\label{def:polytope-R(P,m)}
Let $\mathscr{L}$ be the linearized ample line bundle on $X_P$ corresponding to the given lattice polytope $P\subseteq\mathbb{R}^d$. Let $\mathscr{M}$ be the linearized ample line bundle on $X_{Q(P,m)}$ in Definition~\ref{def:polytope-Q(P,m)}. Consider the projections $p_i\colon X_P^m\times X_{Q(P,m)}\rightarrow X_P$ onto the $i$-th factor for $i=1,\ldots,m$, and $p\colon X_P^m\times X_{Q(P,m)}\rightarrow X_{Q(P,m)}$. Let $\mathscr{A}$ be the ample line bundle on $X_P^m\times X_{Q(P,m)}$ given by
\[
p_1^*\mathscr{L}\otimes\ldots\otimes p_m^*\mathscr{L}\otimes p^*\mathscr{M}.
\]
Let $\mathscr{P}$ be the ample line bundle on $X_{\mathcal{R}(P,m)}$ given by the pullback of $\mathscr{A}$ under the embedding $X_{\mathcal{R}(P,m)}\hookrightarrow X_P^m\times X_{\mathcal{Q}(P,m)}$. Notice that $\mathscr{P}$ comes with a linearization induced by the ones of $\mathscr{L}$ and $\mathscr{M}$. We define $R(P,m)\subseteq(\mathbb{R}^d)^m$ to be the lattice polytope with normal fan $\mathcal{R}(P,m)$ corresponding to the linearized ample line bundle $\mathscr{P}$ on $X_{\mathcal{R}(P,m)}$.
\end{definition}

\begin{theorem}
\label{thm-R(P,m)-is-gen-twisted-Cayley-sum}
The polytope $R(P,m)$ is the generalized $\delta^*$-twisted Cayley sum of polytopes which are normally equivalent to $Q(P,m)$ and with $\delta^*(R(P,m))=mP$.
\end{theorem}

\begin{proof}
As $R(P,m)$ is the lattice polytope corresponding to the polarization induced by $X_{\mathcal{R}(P,m)}\hookrightarrow X_P^m\times X_{Q(P,m)}$, the first step is to explicitly describe $R(P,m)$ in terms of the lattice polytopes $P^m$ and $Q(P,m)$. In the short exact sequence \eqref{s-e-s-N-lattices-tensor-R}, we can factor the map $q$ as follows:
\[
0\rightarrow\mathbb{R}^d\xrightarrow{\;\delta\;}(\mathbb{R}^d)^m\xrightarrow{\;\iota=\id\times q\;}(\mathbb{R}^d)^m\times((\mathbb{R}^d)^m/\mathbb{R}^d)\xrightarrow{\;\pi_2\;}(\mathbb{R}^d)^m/\mathbb{R}^d\rightarrow0.
\]
This gives morphisms of fans $\Sigma_P\xrightarrow{\;\delta\;}\mathcal{R}(P,m)\xrightarrow{\;\iota\;}\Sigma_P^m\times\mathcal{Q}(P,m)$. Dualizing the above sequence and using the identification $((\mathbb{R}^d)^m\times((\mathbb{R}^d)^m/\mathbb{R}^d))^*\cong(\mathbb{R}^d)^{m*}\times((\mathbb{R}^d)^m/\mathbb{R}^d)^*$, we obtain
\[
0\rightarrow((\mathbb{R}^d)^m/\mathbb{R}^d)^*\xrightarrow{\;\pi_2^*\;}(\mathbb{R}^d)^{m*}\times((\mathbb{R}^d)^m/\mathbb{R}^d)^*\xrightarrow{\;\iota^*\;}(\mathbb{R}^d)^{m*}\xrightarrow{\;\delta^*\;}\mathbb{R}^{d*}\rightarrow0.
\]
Hence, the morphism $\iota^*$ sends $P^m\times Q(P,m)$ to the Minkowski sum of $P^m$ and $q^*(Q(P,m))$ (notice that $q^*$ embeds $Q(P,m)$ in $(\mathbb{R}^d)^m$ isomorphically). Additionally, observe that
\[
\iota^*(P^m\times Q(P,m))=R(P,m).
\]
The equality follows by \cite[Lemma~3.1.7]{HLY02} because we fixed compatible linearized ample line bundles on $X_{\mathcal{R}(P,m)}$ and $X_P^m\times X_{\mathcal{Q}(P,m)}$ in Definition~\ref{def:polytope-R(P,m)}. Therefore, we can consider the following diagram:
\begin{center}
\begin{tikzpicture}[>=angle 90]
\matrix(a)[matrix of math nodes,
row sep=2em, column sep=2em,
text height=1.5ex, text depth=0.25ex]
{P^m\times Q(P,m)&R(P,m)\\
P^m&mP,\\};
\path[->] (a-1-1) edge node[above]{$\iota^*$}(a-1-2);
\path[->] (a-1-1) edge node[left]{$p_1$}(a-2-1);
\path[->] (a-2-1) edge node[below]{$s$}(a-2-2);
\path[->] (a-1-2) edge node[right]{$\delta^*$}(a-2-2);
\end{tikzpicture}
\end{center}
where $p_1$ is the projection onto $P^m$ and $s$ is the componentwise sum given by $(x_1,\ldots,x_m)\mapsto x_1+\ldots+x_m$. The diagram is commutative because
\begin{align*}
(\delta^*\circ\iota^*)(x,y)&=\delta^*((\id\times q)^*(x,y))=\delta^*(x+q^*(y))=\delta^*(x)+\delta^*(q^*(y))\\
&=s(x)+(q\circ\delta)^*(y)=s(x)+0=(s\circ p_1)(x,y).
\end{align*}
By the commutativity of the above diagram, the fiber $R(P,m)_v\subseteq R(P,m)$ of $\delta^*$ over $v\in mP$ is equal to $\iota^*(s^{-1}(v)\times Q(P,m))=s^{-1}(v)+q^*(Q(P,m))$.

To conclude it will be enough to show that $R(P,m)_v=s^{-1}(v)+q^*(Q(P,m))$ is normally equivalent to $Q(P,m)$. It is enough to show that the normal fan to $Q(P,m)$ refines the normal fan to every fiber of $s\colon P^m\rightarrow mP$. This follows from the fact that the secondary fan is a fiber fan in the sense of \cite[\S\,2.2]{CM07}.
\end{proof}


\subsection{Combinatorial characterization of the cones in $\mathcal{Q}(P,m)$}

For $X_P=\mathbb{P}^2$, \cite{ST21} gives a combinatorial characterization of the cones in $\mathcal{Q}(P,m)$ in terms of convex subdivisions of $\mathbb{R}^2$. This interpretation generalizes to the case of an arbitrary toric variety $X_P$ as we will soon see. The starting point is the following description of the cones in $\mathcal{Q}(P,m)$, which is discussed in \cite[\S\,1]{KSZ91}. Given a vector $\mathbf{v}=(v_1,\ldots,v_m)\in(\mathbb{R}^d)^m$, define
\[
\Sigma_\mathbf{v}^m=\{\tau\in\Sigma^m\mid\tau\cap(\mathbf{v}+\mathbb{R}^d)\neq\emptyset\},
\]
where $\mathbf{v}+\mathbb{R}^d=\{(v_1+x,\ldots,v_m+x)\mid x\in\mathbb{R}^d\}$. For $\mathbf{v},\mathbf{w}\in(\mathbb{R}^d)^m$, the condition that $\Sigma_\mathbf{v}=\Sigma_\mathbf{w}$ defines an equivalence relation on $(\mathbb{R}^d)^m$, and the images under $q\colon(\mathbb{R}^d)^m\rightarrow(\mathbb{R}^d)^m/\mathbb{R}^d$ of the closure of the corresponding equivalence classes give the cones of the quotient fan $\mathcal{Q}(P,m)$. We now start with our alternative characterization of the cones in $\mathcal{Q}(P,m)$ in terms of convex subdivisions of $\mathbb{R}^d$.

\begin{definition}
\label{def:subdivision-Rd-with-spiders}
Let $P\subseteq\mathbb{R}^d$ be a full dimensional lattice polytope and denote the fan $\Sigma_P$ normal to $P$ by $\Sigma$. Given $v\in\mathbb{R}^d$ denote by $\Sigma(v)$ the fan in $\mathbb{R}^d$ obtained by centering in $-v$ the fan $\Sigma$. In other words, $\Sigma(v)=-v+\Sigma$. The choice of translating by $-v$ is made to be consistent with the conventions adopted in \cite{ST21}.

Given $\mathbf{v}=(v_1,\ldots,v_m)\in(\mathbb{R}^d)^m$, we denote by $\mathcal{S}(\mathbf{v})$ the convex subdivision of $\mathbb{R}^d$ obtained by taking the common refinement of the fans $\Sigma(v_1),\ldots,\Sigma(v_m)$. Note that, in general, the cells in $\mathcal{S}(\mathbf{v})$ may not be cones, and that they are not necessarily bounded. Given $\mathbf{v},\mathbf{w}\in(\mathbb{R}^d)^m$, we say that the subdivisions $\mathcal{S}(\mathbf{v}),\mathcal{S}(\mathbf{w})$ are \emph{equivalent} if given cones $\sigma_i\in\Sigma(v_i),\tau_i\in\Sigma(w_i)$ satisfying $\sigma_i+v_i=\tau_i+w_i$ for all $i$, then $\cap_i\sigma_i\neq\emptyset$ if and only if $\cap_i\tau_i\neq\emptyset$.
\end{definition}

The following lemma is an immediate generalization of \cite[Lemma~8.2]{ST21}.

\begin{lemma}
\label{bijection-cones-subdivisions}
The following statements hold:
\begin{itemize}

\item[(i)] Let $\mathbf{v}\in(\mathbb{R}^d)^m$. The restrictions to $\mathbf{v}+\mathbb{R}^d$ of the cones in $\Sigma_\mathbf{v}^m$ give the subdivision $\mathcal{S}(\mathbf{v})$ under the natural identification $\mathbf{v}+\mathbb{R}^d\cong\mathbb{R}^d$.

\item[(ii)] $[\mathbf{v}],[\mathbf{w}]\in(\mathbb{R}^d)^m/\mathbb{R}^d$ belong to the relative interior of the same cone in the quotient fan $\mathcal{Q}(P,m)$ if and only if $\mathcal{S}(\mathbf{v})$ and $\mathcal{S}(\mathbf{w})$ are equivalent.

\end{itemize}
\end{lemma}

\begin{proof}
If $C\in\mathcal{S}(\mathbf{v})$ is a nonempty cell, then $C=\cap_{i=1}^m\sigma_i$ for some cones $\sigma_i\in\Sigma(v_i)$, $i=1,\ldots,m$. Define
\[
\tau=(v_1+\sigma_1)\times\ldots\times(v_m+\sigma_m)\in\Sigma^m.
\]
Note that $\mathbf{v}+C=(\mathbf{v}+\mathbb{R}^d)\cap\tau$, where $\mathbf{v}+C=\{(v_1+x,\ldots,v_m+x)\mid x\in C\}$. Indeed, an element $\mathbf{w}$ in $\mathbf{v}+C$ can be written as $(v_1+x,\ldots,v_m+x)$ for some $x\in C=\cap_{i=1}^m\sigma_i$. This is equivalent to $v_i+x\in(v_i+\sigma_i)$ for all $i$, which means $\mathbf{w}\in(\mathbf{v}+\mathbb{R}^d)\cap\tau$. As $C\neq\emptyset$, the equality $\mathbf{v}+C=(\mathbf{v}+\mathbb{R}^d)\cap\tau$ guarantees that $\tau\in\Sigma_\mathbf{v}^m$. This discussion proves (i) as it shows that any nonempty cell $C$ arises by intersecting $\mathbf{v}+\mathbb{R}^d$ with a cone in $\Sigma_\mathbf{v}^m$, and conversely that any cone $\tau\in\Sigma_\mathbf{v}^m$ restricts to $\mathbf{v}+\mathbb{R}^d$ giving the cell $\cap_{i=1}^m(-v_i+\tau_i)\in\mathcal{S}(\mathbf{v})$.

Let $\eta\in\Sigma^m$ and let $\sigma_i=\eta_i-v_i$. From part (i), we know that $\mathbf{v}+\cap_i\sigma_i=(\mathbf{v}+\mathbb{R}^d)\cap\eta$. Therefore, $\eta\in\Sigma_\mathbf{v}^m$ if and only if $\cap_i\sigma_i\neq\emptyset$. This shows the claimed equivalence in (ii).
\end{proof}


\subsection{Sections and forgetful morphisms}

We now interpret $X_{\mathcal{R}(P,m)}\rightarrow X_{\mathcal{Q}(P,m)}$ as a family $(m+k)$-pointed toric varieties, where $k$ is the number of torus fixed points of $X_P$. To do this, we need to construct sections for $X_{\mathcal{R}(P,m)}\rightarrow X_{\mathcal{Q}(P,m)}$. The proposition that follows is a generalization of \cite[Proposition~8.9]{ST21}, which was written for $X_P=\mathbb{P}^2$, and the proof is analogous if we use the combinatorial interpretation of the cones in $\mathcal{Q}(P,m)$. Here we propose an alternative direct proof.

\begin{proposition}
\label{sections-generalized}

The following statements hold.

\begin{itemize}

\item[(i)] Let $x_1,\ldots,x_k\in X_P$ be the torus fixed points and for $i\in\{1,\ldots,k\}$ define
\begin{align*}
t_i\colon X_{\mathcal{Q}(P,m)}&\rightarrow X_P^m\times X_{\mathcal{Q}(P,m)},\\
x&\mapsto(x_i,\ldots,x_i,x).
\end{align*}
Then $t_i$ factors through the inclusion $X_{\mathcal{R}(P,m)}\subseteq X_P^m\times X_{\mathcal{Q}(P,m)}$ giving a section of $X_{\mathcal{R}(P,m)}\rightarrow X_{\mathcal{Q}(P,m)}$.

\item[(ii)] For $j\in\{1,\ldots,m\}$, consider the linear map
\begin{align*}
\lambda_j\colon(\mathbb{R}^d)^m/\mathbb{R}^d&\rightarrow(\mathbb{R}^d)^m,\\
[(v_1,\ldots,v_m)]&\mapsto(v_1-v_j,\ldots,v_m-v_j).
\end{align*}
Then $\lambda_j$ induces a morphism of toric varieties $\ell_j\colon X_{\mathcal{Q}(P,m)}\rightarrow X_{\mathcal{R}(P,m)}$ giving a section of $X_{\mathcal{R}(P,m)}\rightarrow X_{\mathcal{Q}(P,m)}$.

\end{itemize}

\end{proposition}

\begin{proof}
To prove (i), we need to show that $\{(x_i,\ldots,x_i)\}\times X_{\mathcal{Q}(P,m)}$ is contained in $X_{\mathcal{R}(P,m)}$. So it is enough to observe that $\{(x_i,\ldots,x_i)\}\times(H^m/H)\subseteq X_{\mathcal{R}(P,m)}$. But the fiber over a point $x$ in the torus $H^m/H$ is a translate by torus action of the diagonally embedded $X_P\subseteq X_P^m$, and therefore such fiber will contain the point $(x_i,\ldots,x_i,x)$ as $x_i$ is torus-invariant.

To prove (ii), let $[\mathbf{v}],[\mathbf{w}]\in(\mathbb{R}^d)^m/\mathbb{R}^d$ in the relative interior of the same cone in $\mathcal{Q}(P,m)$. We want to show that $\lambda_j([\mathbf{v}]),\lambda_j([\mathbf{w}])$ belong to the same cone in $\mathcal{R}(P,m)$. So it suffices to show that they are in the same cone in $\Sigma^m$. We have that $\lambda([\mathbf{v}])=(v_1-v_j,\ldots,v_m-v_j)$, and as $\Sigma^m$ is a complete fan, there exists $\tau=\tau_1\times\ldots\times\tau_m\in\Sigma^m$ such that $(v_1-v_j,\ldots,v_m-v_j)\in\tau$. Moreover, as the $j$-th coordinate of $(v_1-v_j,\ldots,v_m-v_j)$ is zero, we can impose that $\tau_j=0$. As $\tau\cap(\mathbf{v}+\mathbb{R}^d)\neq\emptyset$, we have that $\tau\in\Sigma_\mathbf{v}^m$. Since $[\mathbf{v}],[\mathbf{w}]$ are in the relative interior of the same cone of $\mathcal{Q}(P,m)$, $\Sigma_\mathbf{v}^m=\Sigma_\mathbf{w}^m$, and hence $\tau\in\Sigma_\mathbf{w}^m$. Therefore, $\tau\cap(\mathbf{w}+\mathbb{R}^d)\neq\emptyset$, which means that there exists $\delta\in\mathbb{R}^d$ such that $(w_1+\delta,\ldots,w_m+\delta)\in\tau$. Recall that $\tau_j=0$, which means that $\delta=-w_j$, implying that $(w_1-w_j,\ldots,w_m-w_j)\in\tau$. This proves that $\lambda_j([\mathbf{v}])$ and $\lambda_j([\mathbf{w}])$ are in the same cone $\tau\in\Sigma^m$.
\end{proof}

Finally, the next proposition generalizes \cite[Lemma~8.6]{ST21}. The proof is completely analogous, so we omit it.

\begin{proposition}
\label{forgetful-maps}
Let $i\in\{1,\ldots,m+1\}$ and let $p_i\colon(\mathbb{R}^d)^{m+1}\rightarrow(\mathbb{R}^d)^m$ be the projection
\[
(v_1,\ldots,v_{m+1})\mapsto(v_1,\ldots,\widehat{v}_i,\ldots,v_{m+1}),
\]
Then the induced map $\overline{p}_i\colon(\mathbb{R}^d)^{m+1}/\mathbb{R}^d\rightarrow(\mathbb{R}^d)^m/\mathbb{R}^d$ gives a toric `forgetful' morphism $f_i\colon X_{\mathcal{Q}(P,m+1)}\rightarrow X_{\mathcal{Q}(P,m)}$.
\end{proposition}


\subsection{Flatness and reducedness of the fibers of $X_{\mathcal{R}(P,m)}\rightarrow X_{\mathcal{Q}(P,m)}$}
\label{flatness-and-non-reduced-fibers-example}

For a morphism of toric varieties, flatness and reducedness of the fibers is equivalent to Molcho's weak semistability in \cite{Mol21}. This is the key recurring idea in the results that follow. In particular, Theorem~\ref{thm-criterion-for-reduced-fibers} and Theorem~\ref{P-simplex-flat-reduced-fibers} translate weak semistability into respectively a criterion and an example of flatness and reducedness of the fibers for $X_{\mathcal{R}(P,m)}\rightarrow X_{\mathcal{Q}(P,m)}$.

\begin{proposition}
\label{toric-family-is-equidimensional}
The toric morphism $f\colon X_{\mathcal{R}(P,m)}\rightarrow X_{\mathcal{Q}(P,m)}$ is equidimensional. Moreover, if $X_{\mathcal{Q}(P,m)}$ is smooth or $f$ has reduced fibers, then $f$ is flat.
\end{proposition}

\begin{proof}
Let $\eta\in\mathcal{R}(P,m)$. Then $\eta=q^{-1}(\xi)\cap\tau$ for some $\xi\in\mathcal{Q}(P,m)$ and $\tau\in\Sigma^m$. It follows that
\[
q(\eta)=q(q^{-1}(\xi)\cap\tau)=\xi\cap q(\tau).
\]
Notice that $\xi\cap q(\tau)\in\mathcal{Q}(P,m)$ from the definition of $\mathcal{Q}(P,m)$ as common refinement of the images under $q\colon(\mathbb{R}^d)^m\rightarrow(\mathbb{R}^d)^m/\mathbb{R}^d$ of the cones in $\Sigma^m$. So, the image of every cone in $\mathcal{R}(P,m)$ is a cone in $\mathcal{Q}(P,m)$. Geometrically, this is equivalent to the fact that $\dim(f^{-1}(y))$ is constant for all $y\in X_{\mathcal{Q}(P,m)}$. In other words, the toric morphism $f$ is equidimensional.

Now assume that $X_{\mathcal{Q}(P,m)}$ is smooth. Then $f$ is flat by the miracle flatness theorem \cite[Theorem~23.1]{Mat89} because, as we just proved, $f$ is equidimensional and $X_{\mathcal{R}(P,m)}$ is Cohen--Macaulay (the latter is true for any normal toric variety, see \cite[\S\,2.1]{Ful93}).

Finally, assume that $f$ has reduced fibers. Then the fact that $f$ is equidimensional combined with \cite[Lemma~5.2 and Remark~5.3]{AK00} imply that $f$ is weakly semistable in the sense of \cite[Definition~2.1.2]{Mol21}. Then we can conclude that $f$ is flat by \cite[Theorem~2.1.4]{Mol21}.
\end{proof}

\begin{remark}
\label{on-flatness-of-toric-family}
We can identify several nontrivial examples where $X_{\mathcal{Q}(P,m)}$ is smooth. In \S\,\ref{explicit-examples} we show that $X_{\mathcal{Q}(\Delta_d,2)}$ is smooth for any $d\geq1$. Also, it is well-known that $X_{\mathcal{Q}(\Delta_1,m)}$ is smooth for all $m$ (see \S\,\ref{family-losev-manin-space-revisited}). Other examples where $f\colon X_{\mathcal{R}(P,m)}\rightarrow X_{\mathcal{Q}(P,m)}$ is flat but $X_{\mathcal{Q}(P,m)}$ is not necessarily smooth were found in \cite[Proposition~8.4 and Remark~8.5]{ST21}. These examples have $\dim(P)=2$. In Theorem~\ref{P-simplex-flat-reduced-fibers} we show that $X_{\mathcal{R}(\Delta_d,m)}\rightarrow X_{\mathcal{Q}(\Delta_d,m)}$ is flat with reduced fibers for all $d$ and $m$.
\end{remark}

The next result gives a combinatorial criterion to determine the reducedness of the fibers of $f\colon X_{\mathcal{R}(P,m)}\rightarrow X_{\mathcal{Q}(P,m)}$.

\begin{theorem}
\label{thm-criterion-for-reduced-fibers}
Consider a toric morphism $f\colon X_{\mathcal{R}(P,m)}\rightarrow X_{\mathcal{Q}(P,m)}$. Then $f$ is flat with reduced fibers if and only if for all vectors $\mathbf{v}\in(\mathbb{Z}^d)^m$ and for all non-empty cells $C\in\mathcal{S}(\mathbf{v})$, there exists a point in $C$ with integral coordinates.
\end{theorem}

\begin{proof}
Let us start by assuming that there exists a vector $\mathbf{v}\in(\mathbb{Z}^d)^m$ and a non-empty cell $C\in\mathcal{S}(\mathbf{v})$ with no integral points. We can write $C=\cap_{i=1}^m(-v_i+\tau_i)$ for some $\tau\in\Sigma_P^m$. As $\mathcal{Q}(P,m)$ is a complete fan, there exists a cone $\xi\in\mathcal{Q}(P,m)$ containing $[\mathbf{v}]$. Define $\eta=q^{-1}(\xi)\cap\tau\in\mathcal{R}(P,m)$. We claim that $q\colon(\mathbb{R}^d)^m\rightarrow(\mathbb{R}^d)^m/\mathbb{R}^d$ does not induce a surjection from the lattice points in $\eta$ to the lattice points in $q(\eta)\in\mathcal{Q}(P,m)$. This implies that $X_{\mathcal{R}(P,m)}\rightarrow X_{\mathcal{Q}(P,m)}$ has non-reduced fibers by \cite[Lemma~5.2 and Remark~5.3]{AK00}.

To prove the above claim, first notice that $[\mathbf{v}]$ is a lattice point in $q(\eta)$. This is true because $q(\eta)=\xi\cap q(\tau)$, $[\mathbf{v}]\in\xi$ by definition of $\xi$, and $[\mathbf{v}]\in q(\tau)$ because for any $\delta\in C$, $\mathbf{v}+\delta\in\tau$. Now, assume by contradiction that there exists a lattice point $\mathbf{w}\in\eta$ such that $q(\mathbf{w})=[\mathbf{v}]$. This implies that $\mathbf{w}=\mathbf{v}+z$ for some $z\in\mathbb{Z}^d$. But $\mathbf{v}+z=\mathbf{w}\in\tau$ (recall $\eta=q^{-1}(\xi)\cap\tau$), hence $z$ is contained in the cell $C$. This is a contradiction as $C$ does not contain any integral point.

For the converse, assume that for all vectors $\mathbf{v}\in(\mathbb{Z}^d)^m$ and for all non-empty cells $C\in\mathcal{S}(\mathbf{v})$, there exists a point in $C$ with integral coordinates. To prove that $f$ is flat with reduced fibers, we use \cite[Theorem~2.1.4]{Mol21}. As we already know that $f$ is equidimensional, we only have to show that whenever $\eta\in\mathcal{R}(P,m)$, the lattice points in $\eta$ surject onto the lattice points in $q(\eta)\in\mathcal{Q}(P,m)$.

To prove this, we use the same idea used in \cite[Proposition~8.4]{ST21}. Let $[\mathbf{v}]\in q(\eta)$ be a lattice point. Write $\eta=q^{-1}(\xi)\cap\tau$ for some $\xi\in\mathcal{Q}(P,m)$ and $\tau\in\Sigma_P^m$. This implies that $[\mathbf{v}]\in\xi\cap q(\tau)$, hence $\mathbf{v}\in q^{-1}(\xi)$ and $\mathbf{v}+\delta\in\tau$ for some $\delta\in\mathbb{R}^d$. Let $C\in\mathcal{S}(\mathbf{v})$ be the cell given by $\cap_{i=1}^m(-v_i+\tau_i)$, which is non-empty as $\delta\in C$. By assumption, there exists a lattice point $z\in C$. Finally, notice that $\mathbf{v}+z$ is an integral vector lying in $q^{-1}(\xi)\cap\tau=\eta$ and $q(\mathbf{v}+z)=[\mathbf{v}]$, proving what we needed.
\end{proof}

\begin{example}
\label{ex:non-red-fibers}
A concrete example of family $X_{\mathcal{R}(P,m)}\rightarrow X_{\mathcal{Q}(P,m)}$ with non-reduced fibers is the following. Let $P\subseteq\mathbb{R}^2$ be the lattice polytope
\[
\Conv((0,0),(0,4),(-1,4)).
\]
The corresponding toric surface $X_P$ is isomorphic to $\mathbb{F}_4^0$, which is obtained by contracting the $(-4)$-curve on the Hirzebruch surface $\mathbb{F}_4$. The maximal cones of the normal fan $\Sigma_P$ are given by
\[
\alpha=\langle(-1,0),(0,-1)\rangle,~\beta=\langle(0,-1),(4,1)\rangle,~\gamma=\langle(4,1),(-1,0)\rangle.
\]

Let $\mathbf{w}=(0,0,-3,-1,-5,-2,-6,-3)\in(\mathbb{R}^2)^4$ and consider the cone $\xi\in\mathcal{Q}(P,4)$ containing $[w]$ in its relative interior (see Figure~\ref{fig:non-reduced-fibers} for the corresponding subdivision $\mathcal{S}(\mathbf{w})$ of $\mathbb{R}^2$). Let $\tau\in\Sigma_P^4$ be the cone $\gamma\times\beta\times\beta\times\alpha$. The cell in $\mathcal{S}(\mathbf{w})$ given by
\[
C=((0,0)+\gamma)\cap((3,1)+\beta)\cap((5,2)+\beta)\cap((6,3)+\alpha)
\]
is the shaded $2$-dimensional cell in $\mathcal{S}(\mathbf{w})$ in Figure~\ref{fig:non-reduced-fibers}. As $C$ does not contain integral points, we can conclude by Theorem~\ref{thm-criterion-for-reduced-fibers} that $X_{\mathcal{R}(P,m)}\rightarrow X_{\mathcal{Q}(P,m)}$ has non-reduced fibers.
\end{example}

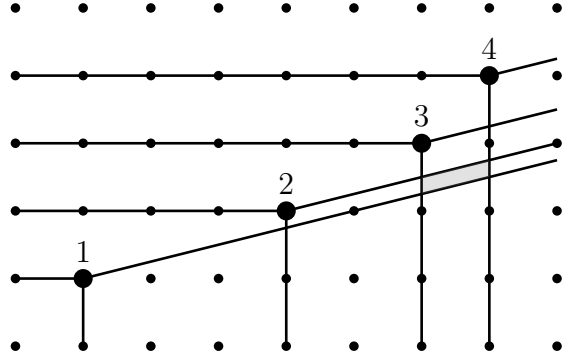
\begin{figure}
\begin{tikzpicture}[scale=0.9]

	\draw[line width=1pt] (0,0) -- (0-1,0);
	\draw[line width=1pt] (0,0) -- (0,0-1);
	\draw[line width=1pt] (0,0) -- (0+7,0+2-1/4);
	\fill (0,0) circle (4pt);

	\draw[line width=1pt] (3,1) -- (3-4,1);
	\draw[line width=1pt] (3,1) -- (3,1-2);
	\draw[line width=1pt] (3,1) -- (3+4,1+1);
	\fill (3,1) circle (4pt);

	\draw[line width=1pt] (5,2) -- (5-6,2);
	\draw[line width=1pt] (5,2) -- (5,2-3);
	\draw[line width=1pt] (5,2) -- (5+2,2+1/2);
	\fill (5,2) circle (4pt);

	\draw[line width=1pt] (6,3) -- (6-7,3);
	\draw[line width=1pt] (6,3) -- (6,3-4);
	\draw[line width=1pt] (6,3) -- (6+1,3+1/4);
	\fill (6,3) circle (4pt);

	\fill (-1,4) circle (2pt);
	\fill (0,4) circle (2pt);
	\fill (1,4) circle (2pt);
	\fill (2,4) circle (2pt);
	\fill (3,4) circle (2pt);
	\fill (4,4) circle (2pt);
	\fill (5,4) circle (2pt);
	\fill (6,4) circle (2pt);
	\fill (7,4) circle (2pt);
	\fill (-1,3) circle (2pt);
	\fill (0,3) circle (2pt);
	\fill (1,3) circle (2pt);
	\fill (2,3) circle (2pt);
	\fill (3,3) circle (2pt);
	\fill (4,3) circle (2pt);
	\fill (5,3) circle (2pt);
	\fill (6,3) circle (2pt);
	\fill (7,3) circle (2pt);
	\fill (-1,2) circle (2pt);
	\fill (0,2) circle (2pt);
	\fill (1,2) circle (2pt);
	\fill (2,2) circle (2pt);
	\fill (3,2) circle (2pt);
	\fill (4,2) circle (2pt);
	\fill (5,2) circle (2pt);
	\fill (6,2) circle (2pt);
	\fill (7,2) circle (2pt);
	\fill (-1,1) circle (2pt);
	\fill (0,1) circle (2pt);
	\fill (1,1) circle (2pt);
	\fill (2,1) circle (2pt);
	\fill (3,1) circle (2pt);
	\fill (4,1) circle (2pt);
	\fill (5,1) circle (2pt);
	\fill (6,1) circle (2pt);
	\fill (7,1) circle (2pt);
	\fill (-1,0) circle (2pt);
	\fill (0,0) circle (2pt);
	\fill (1,0) circle (2pt);
	\fill (2,0) circle (2pt);
	\fill (3,0) circle (2pt);
	\fill (4,0) circle (2pt);
	\fill (5,0) circle (2pt);
	\fill (6,0) circle (2pt);
	\fill (7,0) circle (2pt);
	\fill (-1,-1) circle (2pt);
	\fill (0,-1) circle (2pt);
	\fill (1,-1) circle (2pt);
	\fill (2,-1) circle (2pt);
	\fill (3,-1) circle (2pt);
	\fill (4,-1) circle (2pt);
	\fill (5,-1) circle (2pt);
	\fill (6,-1) circle (2pt);
	\fill (7,-1) circle (2pt);

	\node at (0,0+0.4) {$1$};
	\node at (3,1+0.4) {$2$};
	\node at (5,2+0.4) {$3$};
	\node at (6,3+0.4) {$4$};

\fill[gray!90,nearly transparent] (5,1+1/4) -- (5,1+1/2) -- (6,1+3/4) -- (6,1+1/2) -- cycle;

\end{tikzpicture}
\caption{Subdivision of $\mathbb{R}^2$ giving a cone in $\mathcal{Q}(P,4)$ in Example~\ref{ex:non-red-fibers}.}
\label{fig:non-reduced-fibers}
\end{figure}

We can use Theorem~\ref{thm-criterion-for-reduced-fibers} to identify an important class of families $X_{\mathcal{R}(P,m)}\rightarrow X_{\mathcal{Q}(P,m)}$ which are flat with reduced fibers.

\begin{theorem}
\label{P-simplex-flat-reduced-fibers}
Let $\Delta_d$ be the $d$-dimensional simplex. Then $X_{\mathcal{R}(\Delta_d,m)}\rightarrow X_{\mathcal{Q}(\Delta_d,m)}$ is flat with reduced fibers for all $d\geq1$.
\end{theorem}

\begin{proof}
By Theorem~\ref{thm-criterion-for-reduced-fibers}, it suffices to prove that if $\mathbf{v}\in(\mathbb{Z}^d)^m$ and $C$ is a cell of the subdivision $\mathcal{S}(\mathbf{v})$, then $C$ contains an integral point. We will prove that a finer convex subdivision of $\mathbb{R}^d$ consists of cells which always contain an integral point.

Let $\Sigma$ be the normal fan to $\Delta_d$. The facets of the cones in $\Sigma$ are contained in the hyperplanes $x_i=0$, $x_j-x_h=0$, for $i,j,h\in\{1,\dots,d\}$, $j<h$. Therefore, if we write $\mathbf{v}=(v_1,\ldots,v_m)$ and $v_k=(v_{k1},\ldots,v_{kd})$, then the hyperplanes containing the facets of the cones in $\Sigma(v_k)$ are
\begin{equation}
\label{hyperplanes-subdivision}
x_i=-v_{ki},~x_j-x_h=-v_{kj}+v_{kh}.
\end{equation}
Let $\mathcal{S}(\mathbf{v})'$ be the subdivision of $\mathbb{R}^d$ induced by the above hyperplanes. $\mathcal{S}(\mathbf{v})'$ is a refinement of $\mathcal{S}(\mathbf{v}).$ We now show that for all $C\in\mathcal{S}(\mathbf{v})'$, the vertices of $C$ have integral coordinates.

A vertex $w$ of $C$ can be obtained by intersecting some of the hyperplanes in \eqref{hyperplanes-subdivision}. So $w$ is the unique solution to a system of equations of the form
\begin{equation}
\label{hyperplanes-cutting-vertex}
x_{i_\ell}=\alpha_{i_\ell},~x_{j_s}-x_{h_t}=\beta_{j_sh_t},
\end{equation}
where the constants $\alpha_{i_\ell},\beta_{j_sh_t}$ are integers. We now argue by induction on $d$ that $w$ is an integral vector. If $d=1$, then we can only have one equation $x_1=\alpha_1$, hence $w=w_1=\alpha$ is an integer. Assume the conclusion true for $d\geq1$ and let us prove it for $d+1$. Notice that if \eqref{hyperplanes-cutting-vertex} did not contain any equation of the form $x_{i_\ell}=\alpha_{i_\ell}$ then we could select the highest index $h_t$ and and make $x_{h_t}$ a free parameter, implying that the system will have infinitely many solutions, which is impossible. Therefore, in \eqref{hyperplanes-cutting-vertex} there exists at least one equation of the form $x_{i_\ell}=\alpha_{i_\ell}$. We can then eliminate the variable $x_{i_\ell}$ obtaining a system with $d$ indeterminates and hence conclude the proof by the inductive hypothesis.
\end{proof}

\begin{remark}
The case $d=2$ in Theorem~\ref{P-simplex-flat-reduced-fibers} was known (see \cite[Proposition~8.4]{ST21}). The case $d=1$ can be also viewed as a consequence of the discussion in \S\,\ref{family-losev-manin-space-revisited}.
\end{remark}

\begin{remark}
An example of toric morphism $f\colon X_{\mathcal{R}(P,m)}\rightarrow X_{\mathcal{Q}(P,m)}$ which is flat and with non-reduced fibers will be given in Example~\ref{ex-f-flat-and-with-non-red-fibers}.
\end{remark}


\subsection{Structure of the reducible marked fibers of $X_{\mathcal{R}(P,m)}\rightarrow X_{\mathcal{Q}(P,m)}$}
\label{on-the-structure-of-reducible-fibers}

With the same notation as in \S\,\ref{rel-with-tor-stacks}, consider the toric stack $[V/\!/H]$ and the universal family $\mathcal{U}\rightarrow[V/\!/H]$ with underlying morphism of coarse moduli spaces $U\rightarrow M$. In \cite[Corollary~3.14]{AM16} it is shown that the fibers of $\mathcal{U}\rightarrow[V/\!/H]$ are connected \emph{stable toric varieties} in the sense of Alexeev (for the definition we refer to \cite{Ale02,AB06}). Moreover, the morphism of toric stacks $\mathcal{U}\rightarrow[V/\!/H]$ is the minimal modification of $U\rightarrow M$ which is flat with reduced fibers \cite[\S\,3.5]{Mol14}.

In our case of interest, for a lattice polytope $P$, consider the universal family $\mathcal{U}\rightarrow[X_P^m/\!/H]$ and the underlying morphism of coarse moduli spaces $f\colon X_{\mathcal{R}(P,m)}\rightarrow X_{\mathcal{Q}(P,m)}$. By the above discussion, if $f$ is flat with reduced fibers, then we can conclude that the fibers of $f$ are connected stable toric varieties. The structure of the reducible fibers of $f$ together with the sections can then be further analyzed depending on $P$.

Theorem~\ref{P-simplex-flat-reduced-fibers} guarantees that $f$ is flat with reduced fibers if $P=\Delta_d$. For $d=1$, in the next section we will show that these fibers are exactly the chains of projective lines parametrized by the Losev--Manin moduli space. The marked fibers of $f$ for $d=2$ were studied in \cite{ST21} and correspond to the so called \emph{mixed subdivisions} of $m\Delta_2$ \cite[Definition~1.1]{San05}. If $d\geq3$, then using the theory of Mustafin joins \cite{CHSW11} we have results analogous to \cite[Proposition~8.7 and Corollary~8.10]{ST21} about the marked fibers of $f$.


\section{The universal family of the Losev--Manin moduli space}
\label{family-losev-manin-space-revisited}

The Losev--Manin moduli space $\overline{\mathrm{LM}}_{m+2}$ is the toric variety associated to the permutohedron (see \cite{LM00} and \cite{Kap93}), which corresponds to the fiber polytope associated to the fibration $\Delta_1^m\rightarrow m\Delta_1$ defined by $(x_1,\ldots,x_m)\mapsto x_1+\ldots+x_m$ \cite[Example~5.4]{BS92}. The normal fan to this fiber polytope is exactly the quotient fan $\mathcal{Q}(\Delta_1,m)$, hence $\overline{\mathrm{LM}}_{m+2}\cong X_{\mathcal{Q}(\Delta_1,m)}$. The universal family of $\overline{\mathrm{LM}}_{m+2}$ is given by $\overline{\mathrm{LM}}_{m+3}\rightarrow\overline{\mathrm{LM}}_{m+2}$, where the morphism forgets one of the light points \cite[\S\,2.1]{LM00}. Given the family $X_{\mathcal{R}(\Delta_1,m)}\rightarrow X_{\mathcal{Q}(\Delta_1,m)}$, it is therefore natural to ask whether $\overline{\mathrm{LM}}_{m+3}\cong X_{\mathcal{R}(\Delta_1,m)}$, or equivalently, whether the fan $\mathcal{R}(\Delta_1,m)$ is isomorphic to $\mathcal{Q}(\Delta_1,m+1)$. In the current section we show that this is true. We need a preliminary lemma.

\begin{lemma}
\label{key-lemma-for-family-losev-manin}
For $i\in\{1,\ldots,m+1\}$, consider $\overline{p}_i\colon(\mathbb{R}^d)^{m+1}/\mathbb{R}^d\rightarrow(\mathbb{R}^d)^m/\mathbb{R}^d$. Then
\[
\overline{p}_i(\mathcal{Q}(\Delta_1,m+1))=\mathcal{Q}(\Delta_1,m).
\]
\end{lemma}

\begin{proof}
Without loss of generality assume $i=m+1$. We will use the combinatorial interpretation of the cones in the quotient fan in Lemma~\ref{bijection-cones-subdivisions}. For a cone $\sigma$, denote by $\sigma^\circ$ its relative interior.

Let $\xi\in\mathcal{Q}(\Delta_1,m+1)$ and let us show that $\overline{p}_{m+1}(\xi)\in\mathcal{Q}(\Delta_1,m)$. If $[\mathbf{v}],[\mathbf{w}]\in\xi^\circ$, then the proof of \cite[Lemma~8.6]{ST21} guarantees that $\overline{p}_{m+1}([\mathbf{v}])$ and $\overline{p}_{m+1}([\mathbf{w}])$ lie in the relative interior of a cone $\eta\in\mathcal{Q}(\Delta_1,m)$. Therefore, $\overline{p}_{m+1}(\xi^\circ)\subseteq\eta^\circ$, and we want to show that actually equality holds. Let $[\mathbf{u}]=[u_1,\ldots,u_m]\in\eta^\circ$ and let $[\mathbf{v}]=[v_1,\ldots,v_{m+1}]\in\xi^\circ$. We know that $\mathcal{S}(p_{m+1}(\mathbf{v}))$ is equivalent to $\mathcal{S}(\mathbf{u})$. Consider the unique cell $C\in\mathcal{S}(p_{m+1}(\mathbf{v}))$ such that $-v_{m+1}\in C^\circ$. For $j\in\{1,\ldots,m\}$, we can find cones $\sigma_j\in\Sigma(v_j)$ such that $C=\cap_{j=1}^m\sigma_j$. Let $\tau_j\in\Sigma(u_j)$ such that $\sigma_j+v_j=\tau_j+u_j$, and consider the cell $D=\cap_{j=1}^m\tau_j\in\mathcal{S}(\mathbf{u})$ corresponding to $C$. Let us fix an arbitrary point $-x\in D^\circ$. Then $\mathcal{S}(\mathbf{u},x)$ is equivalent to $\mathcal{S}(\mathbf{v})$ (this is in general false for a polytope $P\neq\Delta_1$, see Remark~\ref{key-lemma-for-family-losev-manin-false-in-general}).

To prove this, let $\alpha_i\in\Sigma(u_i)$, $\alpha\in\Sigma(x)$ and $\beta_i\in\Sigma(v_i)$, $\beta\in\Sigma(v_{m+1})$ such that $\alpha_i+u_i=\beta_i+v_i$ and $\alpha+x=\beta+v_{m+1}$. If $(\cap_i\alpha_i)\cap\alpha\neq\emptyset$, then $\cap_i\alpha_i\neq\emptyset$. As $\mathcal{S}(\mathbf{u})$ is equivalent to $\mathcal{S}(p_{m+1}(\mathbf{v}))$, then also $\cap_i\beta_i\neq\emptyset$. Recall that $D$ is one of the cells of the subdivision $\mathcal{S}(\mathbf{u})$ of $\mathbb{R}^1$, and the same applies to $\cap_i\alpha_i$. Therefore, either $D=\cap_i\alpha_i$ or they have disjoint interiors. Let us discuss these possibilities individually.
\begin{itemize}

\item Assume that $\cap_i\alpha_i=D$. Then $\cap_i\beta_i=C$, which implies that $(\cap_i\beta_i)\cap\beta\neq\emptyset$ because $-v_{m+1}\in(\cap_i\beta_i)\cap\beta$.

\item Now suppose that $\cap_i\alpha_i$ and $D$ have disjoint interiors. Then $D$ is either on the left or right of $\cap_i\alpha_i$ with respect to the orientation given by the real line $\mathbb{R}$. Without loss of generality, say $D$ lies on the left of $\cap_i\alpha_i$. As $(\cap_i\alpha_i)\cap\alpha\neq\emptyset$, we must have that $\alpha=-x+\mathbb{R}_{\geq0}$. So also $\beta=-v_{m+1}+\mathbb{R}_{\geq0}$, and as $\mathcal{S}(\mathbf{u})$ is equivalent to $\mathcal{S}(\mathbf{v})$, then also $C$ lies on the left of $\cap_i\beta_i$. Hence $(\cap_i\beta_i)\cap\beta\neq\emptyset$.

\end{itemize}
We can argue that $(\cap_i\beta_i)\cap\beta\neq\emptyset$ implies $(\cap_i\alpha_i)\cap\alpha\neq\emptyset$ in an analogous way. As we showed that $\mathcal{S}(\mathbf{u},x)$ is equivalent to $\mathcal{S}(\mathbf{v})$, we can conclude that $[\mathbf{u},x]\in\xi^\circ$. In conclusion, $[\mathbf{u}]\in\overline{p}_{m+1}(\xi^\circ)$, so $\overline{p}_{m+1}(\xi^\circ)=\eta^\circ$. This shows that $\overline{p}_{m+1}(\mathcal{Q}(\Delta_1,m+1))\subseteq\mathcal{Q}(\Delta_1,m)$. Moreover, this is enough to conclude that $\overline{p}_{m+1}(\mathcal{Q}(\Delta_1,m+1))$ and $\mathcal{Q}(\Delta_1,m)$ are equal because they are both complete fans.
\end{proof}

\begin{remark}
\label{key-lemma-for-family-losev-manin-false-in-general}
The above lemma is in general false for polytopes $P$ different from $\Delta_1$. More precisely, if $\xi\in\mathcal{Q}(P,m+1)$, then $\overline{p}_i(\xi)$ is not necessarily a cone in $\mathcal{Q}(P,m)$. For example, let $P=\Delta_2$ and consider the cone $\xi\in\mathcal{Q}(\Delta_2,4)$ (resp. $\eta\in\mathcal{Q}(\Delta_2,3)$) corresponding to the subdivision of $\mathbb{R}^2$ in Figure~\ref{example-failure-key-lemma}~(1) (resp. (2)). Then $\overline{p}_4(\xi^\circ)\subsetneq\eta^\circ$, and an explicit example of $[\mathbf{u}]\in\eta^\circ\setminus\overline{p}_4(\xi^\circ)$ is the vector corresponding to Figure~\ref{example-failure-key-lemma}~(3). To prove this, suppose by contradiction that $\overline{p}_4([\mathbf{v}])=[p_4(\mathbf{v})]=[\mathbf{u}]$ for some $\mathbf{v}\in\xi^\circ$. Then, up to changing the representative for the class $[\mathbf{v}]$, we can assume that $p_4(\mathbf{v})=\mathbf{u}$. However, there is no choice of $v_4\in\mathbb{R}^2$ such that the subdivision $\mathcal{S}(\mathbf{u},v_4)=\mathcal{S}(\mathbf{v})$ is equivalent to the one in Figure~\ref{example-failure-key-lemma}~(1), creating a contradiction with the fact that $\mathbf{v}\in\xi^\circ$.
\end{remark}

\begin{figure}
\begin{tikzpicture}[scale=0.7]

	\draw[line width=1pt] (0,0) -- (0+5,0);
	\draw[line width=1pt] (0,0) -- (0,0+2);
	\draw[line width=1pt] (0,0) -- (0-1,0-1);
	\draw[dashed,line width=1pt] (0,0) -- (0+5,0+5);
	\fill (0,0) circle (4pt);

	\draw[line width=1pt] (1,2) -- (1+5,2);
	\draw[line width=1pt] (1,2) -- (1,2+1);
	\draw[line width=1pt] (1,2) -- (1-2,2-2);
	\fill (1,2) circle (4pt);

	\draw[line width=1pt] (5,0) -- (5+1,0);
	\draw[line width=1pt] (5,0) -- (5,0+5);
	\draw[line width=1pt] (5,0) -- (5-1,0-1);
	\fill (5,0) circle (4pt);

	\draw[line width=1pt] (4,3) -- (4+2,3);
	\draw[line width=1pt] (4,3) -- (4,3+2);
	\draw[line width=1pt] (4,3) -- (4-4,3-4);
	\fill (4,3) circle (4pt);

	\node at (0,-0.5) {$1$};
	\node at (1-0.5,2) {$2$};
	\node at (5,-0.5) {$3$};
	\node at (4-0.5,3) {$4$};

	\node at (2.5,-1.5) {(1)};

\end{tikzpicture}
\begin{tikzpicture}[scale=0.7]

	\draw[line width=1pt] (0,0) -- (0+5,0);
	\draw[line width=1pt] (0,0) -- (0,0+2);
	\draw[line width=1pt] (0,0) -- (0-1,0-1);
	\draw[dashed,line width=1pt] (0,0) -- (0+3,0+3);
	\fill (0,0) circle (4pt);

	\draw[line width=1pt] (1,2) -- (1+5,2);
	\draw[line width=1pt] (1,2) -- (1,2+1);
	\draw[line width=1pt] (1,2) -- (1-2,2-2);
	\fill (1,2) circle (4pt);

	\draw[line width=1pt] (5,0) -- (5+1,0);
	\draw[line width=1pt] (5,0) -- (5,0+3);
	\draw[line width=1pt] (5,0) -- (5-1,0-1);
	\fill (5,0) circle (4pt);

	\node at (0,-0.5) {$1$};
	\node at (1-0.5,2) {$2$};
	\node at (5,-0.5) {$3$};

	\node at (2.5,-1.5) {(2)};

\end{tikzpicture}
\begin{tikzpicture}[scale=0.7]

	\draw[line width=1pt] (0,0) -- (0+2,0);
	\draw[line width=1pt] (0,0) -- (0,0+2);
	\draw[line width=1pt] (0,0) -- (0-1,0-1);
	\draw[dashed,line width=1pt] (0,0) -- (0+3,0+3);
	\fill (0,0) circle (4pt);

	\draw[line width=1pt] (1,2) -- (1+2,2);
	\draw[line width=1pt] (1,2) -- (1,2+1);
	\draw[line width=1pt] (1,2) -- (1-2,2-2);
	\fill (1,2) circle (4pt);

	\draw[line width=1pt] (2,0) -- (2+1,0);
	\draw[line width=1pt] (2,0) -- (2,0+3);
	\draw[line width=1pt] (2,0) -- (2-1,0-1);
	\fill (2,0) circle (4pt);

	\node at (0,-0.5) {$1$};
	\node at (1-0.5,2) {$2$};
	\node at (2,-0.5) {$3$};

	\node at (1,-1.5) {(3)};

\end{tikzpicture}
\caption{Example in Remark~\ref{key-lemma-for-family-losev-manin-false-in-general} showing that Lemma~\ref{key-lemma-for-family-losev-manin} does not hold in general for $P\neq\Delta_1$.}
\label{example-failure-key-lemma}
\end{figure}
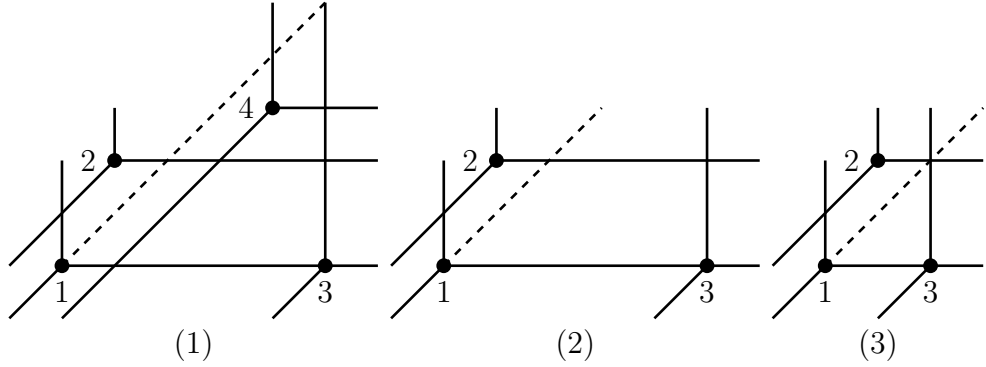

\begin{theorem}
\label{thm-losev-manin-family-revisited}
For $i=1,\ldots,m+1$, consider the diagram
\begin{center}
\begin{tikzpicture}[>=angle 90]
\matrix(a)[matrix of math nodes,
row sep=2em, column sep=2em,
text height=1.5ex, text depth=0.25ex]
{(\mathbb{R}^1)^{m+1}&(\mathbb{R}^1)^m\\
(\mathbb{R}^1)^{m+1}/\mathbb{R}^1&(\mathbb{R}^1)^m/\mathbb{R}^1,\\};
\path[->] (a-1-1) edge node[above]{$p_i$}(a-1-2);
\path[->] (a-1-1) edge node[left]{$q$}(a-2-1);
\path[->] (a-2-1) edge node[below]{$\overline{p}_i$}(a-2-2);
\path[->] (a-1-2) edge node[right]{$q$}(a-2-2);
\path[->] (a-2-1) edge node[above left]{$f_i$}(a-1-2);
\end{tikzpicture}
\end{center}
where $f_i([\mathbf{v}])=(v_1-v_i,\ldots,\widehat{v_i-v_i},\ldots,v_{m+1}-v_i)$. Then $f_i$ induces an isomorphism of fans $\mathcal{Q}(\Delta_1,m+1)\cong\mathcal{R}(\Delta_1,m)$.
\end{theorem}

\begin{proof}
(Note that $q\circ p_i=\overline{p}_i\circ q$, $\overline{p}_i=q\circ f_i$, but $p_i\neq f_i\circ q$.) Without loss of generality, let $i=m+1$. We start with a preliminary construction. Let $\xi\in\mathcal{Q}(\Delta_1,m+1)$. Let $\mathbf{v}$ be a vector in the relative interior $\xi^\circ$ of $\xi$. For each $i=1,\ldots,m$ define $\tau_{\xi,i}^\mathbf{v}\in\Sigma$ as follows: if $v_{m+1}=v_i$, let $\tau_{\xi,i}^\mathbf{v}=\{0\}$, otherwise let $\tau_{\xi,i}^\mathbf{v}$ be the unique half-line such that $-v_{m+1}\in-v_i+\tau_{\xi,i}^\mathbf{v}$. Define $\tau_\xi^\mathbf{v}=\tau_{\xi,1}^\mathbf{v}\times\ldots\times\tau_{\xi,m}^\mathbf{v}$. The cone $\tau_\xi^\mathbf{v}$ has the following properties.

\begin{enumerate}

\item $\tau_\xi^\mathbf{v}$ is independent from the choice of $\mathbf{v}\in\xi^\circ$. To show this, let $\mathbf{w}$ be another vector in $\xi^\circ$. We want to show that $\tau_{\xi,i}^\mathbf{v}=\tau_{\xi,i}^\mathbf{w}$ for $i=1,\ldots,m$. The fact that $\mathcal{S}(\mathbf{v})$ is equivalent to $\mathcal{S}(\mathbf{w})$ implies that that $v_i=v_{m+1}$ if and only if $w_i=w_{m+1}$. In this case, $\tau_{\xi,i}^\mathbf{v}=\{0\}=\tau_{\xi,i}^\mathbf{w}$. So let us assume that $v_i\neq v_{m+1}$ and $w_i\neq w_{m+1}$. In this case, both $\tau_{\xi,i}^\mathbf{v},\tau_{\xi,i}^\mathbf{w}$ are half-lines in $\mathbb{R}^1$ centered at $0$. We need to show they are both equal to $\mathbb{R}_{\geq0}$ or $\mathbb{R}_{\leq0}$. Without loss of generality, assume that $\tau_{\xi,i}^\mathbf{v}=\mathbb{R}_{\geq0}$. As $\mathcal{S}(\mathbf{v})$ is equivalent to $\mathcal{S}(\mathbf{w})$ and $\{-v_{m+1}\}\cap(-v_i+\mathbb{R}_{\geq0})=\{-v_{m+1}\}\neq\emptyset$, then $\{-w_{m+1}\}\cap(-w_i+\mathbb{R}_{\geq0})\neq\emptyset$ as well, so $\tau_{\xi,i}^\mathbf{w}=\mathbb{R}_{\geq0}$ by definition. As we showed $\tau_{\xi,i}^\mathbf{v}=\tau_{\xi,i}^\mathbf{w}$ for any $\mathbf{v},\mathbf{w}\in\xi^\circ$, we will denote this cone simply be $\tau_\xi$.

\item A simple consequence of the above property which we will use is the following. If $\mathbf{w}\in\xi^\circ$, then $-w_{m+1}\in-w_i+\tau_{\xi,i}$ for $i=1,\ldots,m$.

\item If $\mathbf{w}\in\xi^\circ$, consider the cell $C=\cap_{i=1}^m(-w_i+\tau_{\xi,i})\in\mathcal{S}(p_{m+1}(\mathbf{w}))$. Then $C=\{-w_{m+1}\}$ or $\{-w_{m+1}\}\subsetneq C^\circ$: the first case occurs if $w_{m+1}=w_i$ for some $i\in\{1,\ldots,m\}$, otherwise we have the latter.

\end{enumerate}

By definition $q^{-1}(\overline{p}_{m+1}(\xi))\cap\tau_\xi$ is a cone in $\mathcal{R}(\Delta_1,m)$ since $\overline{p}_{m+1}(\xi)\in\mathcal{Q}(\Delta_1,m)$ by Lemma~\ref{key-lemma-for-family-losev-manin}. We claim that
\[
f_{m+1}(\xi)=q^{-1}(\overline{p}_{m+1}(\xi))\cap\tau_\xi.
\]
First of all, the commutativity of the lower-right triangle of the above diagram yields $q(f_{m+1}(\xi))=\overline{p}_{m+1}(\xi)$, so $f_{m+1}(\xi)\subseteq q^{-1}(\overline{p}_{m+1}(\xi))$. Moreover, for any $\mathbf{v}\in\xi$, $f_{m+1}(\mathbf{v})=(v_1-v_{m+1},\ldots,v_m-v_{m+1})$, and we know that $-v_{m+1}\in-v_i+\tau_{\xi,i}$ for all $i=1,\ldots,m$. So $(v_1-v_{m+1},\ldots,v_m-v_{m+1})\in\tau_\xi$. This shows that $f_{m+1}(\xi)\subseteq\tau_\xi$, and hence $f_{m+1}(\xi)\subseteq q^{-1}(\overline{p}_{m+1}(\xi))\cap\tau_\xi$.

To prove the second containment, it will be enough to show that $q^{-1}(\overline{p}_{m+1}(\xi^\circ))\cap\tau_\xi\subseteq f_{m+1}(\xi)$. Let $(v_1,\ldots,v_m)\in q^{-1}(\overline{p}_{m+1}(\xi^\circ))\cap\tau_\xi$. So $[v_1,\ldots,v_m]\in\overline{p}_{m+1}(\xi^\circ)$, hence we can write $[v_1,\ldots,v_m]=\overline{p}_{m+1}([w_1,\ldots,w_{m+1}])=[w_1,\ldots,w_m]$ for some $[w_1,\ldots,w_{m+1}]\in\xi^{\circ}$. So there exists $\delta\in\mathbb{R}^1$ such that $v_i+\delta=w_i$ for $i=1,\ldots,m$. As $[w_1,\ldots,w_{m+1}]\in\xi^\circ$, we have that $-w_{m+1}\in-w_i+\tau_{\xi,i}$ for $i=1,\ldots,m$ by (2) above. Notice that $(v_1,\ldots,v_m)\in\tau_\xi$ by assumption, so $0\in-v_i+\tau_{\xi,i}$ for all $i=1,\ldots,m$, and hence $-\delta\in-\delta-v_i+\tau_{\xi,i}=-w_i+\tau_{\xi,i}$. Let $\cap_{i=1}^m(-w_i+\tau_{\xi,i})=C$, which by (3) above is equal to $\{-w_{m+1}\}$ or $\{-w_{m+1}\}\subsetneq C^\circ$. As $-\delta\in C$, three possibilities can occur:
\begin{itemize}
\item If $C=\{-w_{m+1}\}$, then $\delta=w_{m+1}$, hence $[w_1,\ldots,w_m,\delta]\in\xi^\circ$. So in what follows assume that $\{-w_{m+1}\}\subsetneq C^\circ$.

\item If $-\delta\in C^\circ$, then $[w_1,\ldots,w_m,\delta]\in\xi^\circ$ because $\mathcal{S}(w_1,\ldots,w_m,\delta)$ is equivalent to $\mathcal{S}(\mathbf{w})$. (Here we used that we are working with $\Delta_1$. This is not true if we take higher dimensional polytopes. For instance, if $P=\Delta_2$ and $m=1$, let $v=(0,0),x=(-1,-2),y=(-2,-1)$. Although $-x,-y$ belong to the relative interior of the same cell in $\mathcal{S}(v)$, we have that $\mathcal{S}(v,x)$ is not equivalent to $\mathcal{S}(v,y)$.)

\item If $-\delta\in C\setminus C^\circ$, then $[w_1,\ldots,w_m,\delta]$ belongs to a face of $\xi$.

\end{itemize}
In either case, we can conclude that $[w_1,\ldots,w_m,\delta]\in\xi$. Since $f_{m+1}([w_1,\ldots,w_m,\delta])=(w_1-\delta,\ldots,w_m-\delta)=(v_1,\ldots,v_m)$, we have that $(v_1,\ldots,v_m)\in f_{m+1}(\xi)$, proving that $q^{-1}(\overline{p}_{m+1}(\xi^\circ))\cap\tau_\xi\subseteq f_{m+1}(\xi)$.

So far we proved that $f_{m+1}(\mathcal{Q}(\Delta_1,m+1))\subseteq\mathcal{R}(\Delta_1,m)$. As $f_{m+1}$ is an isomorphism and $\mathcal{Q}(\Delta_1,m+1)$ is a complete fan, then $f_{m+1}(\mathcal{Q}(\Delta_1,m+1))$ is a complete fan, hence $f_{m+1}(\mathcal{Q}(\Delta_1,m+1))$ and $\mathcal{R}(\Delta_1,m)$ coincide. The fact that $f_{m+1}$ is an isomorphism also implies that we cannot have different cones in $\mathcal{Q}(\Delta_1,m+1))$ mapping to the same cone in $\mathcal{R}(\Delta_1,m)$. In conclusion, $f_{m+1}$ induces the claimed isomorphism of fans.
\end{proof}

\begin{corollary}
For $m\geq1$, the $m$-dimensional permutohedron $Q(\Delta_1,m+1)$ is a generalized $\pi$-twisted Cayley sum over $m\Delta_1$ of $m+1$ polytopes $R_1,\ldots,R_{m+1}$ normally equivalent to $Q(\Delta_1,m)$ such that $\{\pi(R_1),\ldots,\pi(R_{m+1})\}=m\Delta_1\cap\mathbb{Z}$.
\end{corollary}

\begin{proof}
Combining Theorem~\ref{thm-R(P,m)-is-gen-twisted-Cayley-sum} with Theorem~\ref{thm-losev-manin-family-revisited} we have that $Q(\Delta_1,m+1)$ is a generalized $\pi$-twisted Cayley sum over $m\Delta_1$ of $k$ polytopes $R_1,\ldots,R_k$ normally equivalent to $Q(\Delta_1,m)$. As $m\Delta_1$ contains $m+1$ lattice points, we have that $k\leq m+1$. On the other hand, $Q(\Delta_1,m+1)$ has $(m+1)!$ vertices, so $(m+1)!\leq k\cdot m!$, where $m!$ is equal to the number of vertices of $R_i$, $i=1,\ldots,k$. This implies that $k=m+1$.
\end{proof}


\section{Explicit examples}
\label{explicit-examples}

It is known that $\mathcal{Q}(\Delta_1,m)$ is a smooth fan. We will see that this property extends to $\mathcal{Q}(\Delta_d,2)$, where $\Delta_d$ is the $d$-dimensional simplex. Recall that the smoothness of $\mathcal{Q}(P,m)$ implies that $X_{\mathcal{R}(P,m)}\rightarrow X_{\mathcal{Q}(P,m)}$ is flat by Proposition~\ref{toric-family-is-equidimensional}. We start with the following general lemma.

\begin{lemma}
\label{explicit-quotient-fan-for-m=2}
Let $P\subseteq\mathbb{R}^d$ be a lattice polytope and consider the quotient fan $\mathcal{Q}(P,2)$ in $(\mathbb{R}^d)^2/\mathbb{R}^d$. Consider the identification $(\mathbb{R}^d)^2/\mathbb{R}^d\cong\mathbb{R}^d$ given by $[v,w]\mapsto v-w$. Then $\mathcal{Q}(P,2)$ equals the common refinement of $\Sigma_P$ and $-\Sigma_P$.
\end{lemma}

\begin{proof}
Denote by $\mathcal{F}$ the common refinement of $\Sigma_P$ and $-\Sigma_P$. Note that
\[
\mathcal{F}=\{\tau_1\cap(-\tau_2)\mid\tau_1,\tau_2\in\Sigma_P\}.
\]
We show that $\mathcal{Q}(P,2)$ and $\mathcal{F}$ refine each other. First of all, $q(\Sigma_P\times\{0\})=\Sigma_P$ and $q(\{0\}\times\Sigma_P)=-\Sigma_P$. So, as the cones in $\Sigma_P$ and $-\Sigma_P$ are some of the images that arise in $q(\Sigma_P^2)$, we have that $\mathcal{Q}(P,2)$ refines $\mathcal{F}$. Conversely, let $\tau_1,\tau_2\in\Sigma_P$. Then
\[
q(\tau_1\times\tau_2)=\tau_1-\tau_2,
\]
which contains both $\tau_1$ and $-\tau_2$. So the cone in $\mathcal{F}$ given by $\tau_1\cap(-\tau_2)$ is contained in $q(\tau_1\times\tau_2)$. This implies that $\mathcal{F}$ is a fan that refines the images in $q(\Sigma_P^2)$. As the fan $\mathcal{Q}(P,2)$ is the common refinement of the images in $q(\Sigma_P^2)$, we must have that $\mathcal{F}$ refines $\mathcal{Q}(P,2)$ as well.
\end{proof}

\begin{example}
Let $P$ be the lattice polytope $\Conv((0,0),(3,0),(1,1),(0,1))$, so that $X_P$ is isomorphic to the Hirzebruch surface $\mathbb{F}_2$. The normal fan $\Sigma_{P}$ has rays given by $e_1$, $e_2$, $-e_2$, and $-e_1-2e_2$. Then $\mathcal{Q}(P,2)$ is not a smooth fan, as it contains the singular cone generated by $e_1,e_1+2e_2$.
\end{example}

\begin{example}
\label{ex-f-flat-and-with-non-red-fibers}
Let $\ell_1,\ell_2\subseteq\mathbb{P}^1\times\mathbb{P}^1$ be two incident rulings away from the torus fixed points $p_1,\ldots,p_4\in\mathbb{P}^1\times\mathbb{P}^1$. Let $Y$ be the blow up of $\mathbb{P}^1\times\mathbb{P}^1$ at $p_1,\ldots,p_4$ and denote by $E_i\subseteq Y$ the exceptional divisor over $p_i$. Let $P$ be the lattice polytope corresponding to the toric variety $Y$ endowed with the ample polarization $3\ell_1+3\ell_2-E_1-\ldots-E_4$. As $\Sigma_P=-\Sigma_P$, by Lemma~\ref{explicit-quotient-fan-for-m=2} we have that $X_{\mathcal{Q}(P,2)}\cong X_P$, which is smooth. Hence, the toric morphism $f\colon X_{\mathcal{R}(P,2)}\rightarrow X_{\mathcal{Q}(P,2)}$ is flat by Proposition~\ref{toric-family-is-equidimensional}. On the other hand, by Theorem~\ref{thm-criterion-for-reduced-fibers} we have that $f$ has non-reduced fibers. This is because if $\mathbf{v}=(0,0,-1,0)\in(\mathbb{Z}^2)^2$, then $\mathcal{S}(\mathbf{v})$ has a cell consisting of the single point $\left(\frac{1}{2},\frac{1}{2}\right)$.
\end{example}

\begin{corollary}
\label{blow-up-torus-fixed-poins-projective-space}
Let $\Delta_d$ be the $d$-dimensional simplex. Then the toric variety $X_{\mathcal{Q}(\Delta_d,2)}$ is isomorphic to the blow up of $\mathbb{P}^d$ at the $d+1$ torus fixed points. In particular, $X_{\mathcal{Q}(\Delta_d,2)}$ is smooth.
\end{corollary}

\begin{proof}
Recall that given a cone $\sigma$ generated by rays $r_1,\ldots,r_\ell$, the simple blow up of the associated toric variety $X_\sigma$ at the torus fixed point corresponding to $\sigma$ is the toric variety associated to the fan obtained by taking the star subdivision of $\sigma$, which inserts the ray $r_1+\ldots+r_\ell$.

Now consider $P=\Delta_d\subseteq\mathbb{R}^d$ and let $e_1,\ldots,e_d$ be the canonical basis of $\mathbb{R}^d$. The maximal dimensional cones of $\Sigma_{\Delta_d}$ are
\[
\langle e_1,\ldots,e_d\rangle,~\langle e_1,\ldots,\widehat{e}_i,\ldots,e_d,-(e_1+\ldots+e_d)\rangle,~i=1,\ldots,d.
\]
The star subdivision of $\langle e_1,\ldots,e_d\rangle$ inserts the ray $e_1+\ldots+e_d$ and the star subdivision of $\langle e_1,\ldots,\widehat{e}_i,\ldots,e_d,-e_1-\ldots-e_d\rangle$ inserts the ray $-e_i$. Therefore, the common refinement of $\Sigma_{\Delta_d}$ and $-\Sigma_{\Delta_d}$, which equals $\mathcal{Q}(\Delta_d,2)$ by Lemma~\ref{explicit-quotient-fan-for-m=2}, gives us exactly the blow up of $\mathbb{P}^d$ at its $d+1$ torus fixed points.
\end{proof}

\begin{example}
If $P=\Delta_1$, then $Q(P,m)$ is the permutohedron, and $R(P,m)=Q(P,m+1)$. So the simplest new example to consider is $R(\Delta_2,2)$. In what follows, we explicitly list all the maximal cones of the fan $\mathcal{R}(\Delta_2,2)$. Let us fix an isomorphism $(\mathbb{R}^2)^2/\mathbb{R}^2\cong\mathbb{R}^2$ so that $q\colon(\mathbb{R}^2)^2\rightarrow(\mathbb{R}^2)^2/\mathbb{R}^2\cong\mathbb{R}^2$ becomes:
\[
q((x,y),(u,v))=(x-u,y-v)=(a,b).
\]
Then by Lemma~\ref{explicit-quotient-fan-for-m=2} we have that $\mathcal{Q}(\Delta_2,2)$ is given in Figure~\ref{quotient-fan-Delta2-2}. We now consider all possible combinations $q^{-1}(\xi)\cap(\tau_1\times\tau_2)$, $\xi\in\mathcal{Q}(\Delta_2,2)$ and $\tau_1,\tau_2\in\Sigma_{\Delta_2}$, which give rise to maximal dimensional cones, and we compute the generating extremal rays using Macaulay2 \cite{M2,PolyhedraSource}. The final result can be found in Table~\ref{max-cones-fan-R-Delta2-2}. As an example, consider
\[
\xi=\{(a,b)\mid a\geq0,a-b\geq0\},~\tau_1=\{(x,y)\mid x\geq0,y\geq0\},~\tau_2=\{(u,v)\mid u\geq0,v\geq0\}.
\]
Then the cone $q^{-1}(\xi)\cap(\tau_1\times\tau_2)$ is cut out by the following inequalities:
\[
x-u\geq0,~x-y-u+v\geq0,~x\geq0,~y\geq0,~u\geq0,~v\geq0.
\]
This cone has $(0,0,0,1),(0,1,0,1),(1,0,0,0),(1,0,1,0),(1,1,0,0)$ as generators of the extremal rays, which is case $1$ in Table~\ref{max-cones-fan-R-Delta2-2}.
\end{example}

\begin{figure}
\begin{tikzpicture}[scale=0.6]

	\draw[->,line width=1pt] (-3,0) -- (3,0);
	\draw[->,line width=1pt] (0,-3) -- (0,3);
	\draw[line width=1pt] (-2.6,-2.6) -- (2.6,2.6);

	\node at (4,1) {$a\geq0,~a-b\geq0$};
	\node at (3,3.5) {$a\geq0,~-a+b\geq0$};
	\node at (-3,2) {$-a\geq0,~b\geq0$};
	\node at (-4.5,-1) {$-a\geq0,~-a+b\geq0$};
	\node at (-3,-3.5) {$-a\geq0,~a-b\geq0$};
	\node at (3,-2) {$a\geq0,~-b\geq0$};

	\node at (3.3,-0.3) {$a$};
	\node at (-0.3,3.3) {$b$};

\end{tikzpicture}
\caption{The quotient fan $\mathcal{Q}(\Delta_2,2)$.}
\label{quotient-fan-Delta2-2}
\end{figure}
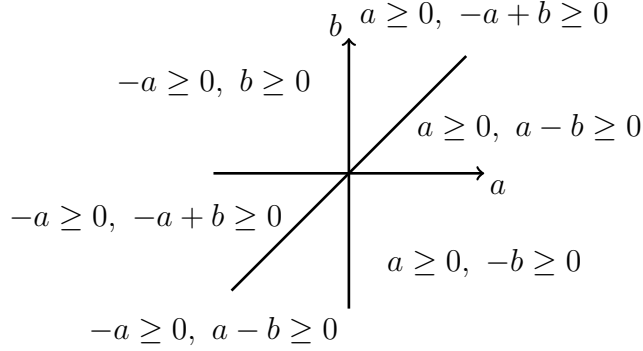

\begin{table}
\centering
\caption{List of maximal cones in $\mathcal{R}(\Delta_2,2)$.}
\label{max-cones-fan-R-Delta2-2}
\begin{tabular}{|>{\centering\arraybackslash}m{0.4cm}|>{\centering\arraybackslash}m{7cm}||>{\centering\arraybackslash}m{0.4cm}|>{\centering\arraybackslash}m{7cm}|}
\hline
$1$ & \texttt{0001,0101,1000,1010,1100}
&
$19$ & \texttt{-1-1-1-1,-1000,000-1,0010,0100}
\\
\hline
$2$ & \texttt{00-1-1,0001,0101,1000,1100}
&
$20$ & \texttt{0010,0011,0100,0101,1010}
\\
\hline
$3$ & \texttt{00-1-1,1000,1010,1100}
&
$21$ & \texttt{-1-100,0010,0011,1010}
\\
\hline
$4$ & \texttt{0-100,0001,1000,1010}
&
$22$ & \texttt{-1-1-1-1,-1-100,000-1,0010,1010}
\\
\hline
$5$ & \texttt{-1-1-1-1,0-100,00-1-1,0001,1000}
&
$23$ & \texttt{-1-100,0010,0011,0100,0101}
\\
\hline
$6$ & \texttt{-1-1-1-1,0-100,00-1-1,1000,1010}
&
$24$ & \texttt{-1-1-1-1,-1-100,0100,0101}
\\
\hline
$7$ & \texttt{-1-1-1-1,00-1-1,0001,0101}
&
$25$ & \texttt{-1-1-1-1,-1-100,000-1,0010,0100}
\\
\hline
$8$ & \texttt{0100,0101,1010,1100}
&
$26$ & \texttt{0001,0011,0101,1010}
\\
\hline
$9$ & \texttt{00-1-1,0100,0101,1100}
&
$27$ & \texttt{-1-100,0-100,0001,0011,1010}
\\
\hline
$10$ & \texttt{00-1-1,000-1,0100,1010,1100}
&
$28$ & \texttt{-1-1-1-1,-1-100,0-100,0001}
\\
\hline
$11$ & \texttt{-1-1-1-1,00-1-1,000-1,1010}
&
$29$ & \texttt{-1-1-1-1,-1-100,0-100,1010}
\\
\hline
$12$ & \texttt{-1-1-1-1,00-1-1,0100,0101}
&
$30$ & \texttt{-1-100,0001,0011,0101}
\\
\hline
$13$ & \texttt{-1-1-1-1,00-1-1,000-1,0100}
&
$31$ & \texttt{-1-1-1-1,-1-100,0001,0101}
\\
\hline
$14$ & \texttt{0010,0100,0101,1010}
&
$32$ & \texttt{0001,0101,1000,1010}
\\
\hline
$15$ & \texttt{000-1,0010,0100,1010}
&
$33$ & \texttt{00-10,0001,0101,1000}
\\
\hline
$16$ & \texttt{-1-1-1-1,000-1,0010,1010}
&
$34$ & \texttt{-1-1-1-1,0-100,00-10,0001,1000}
\\
\hline
$17$ & \texttt{-1000,0010,0100,0101}
&
$35$ & \texttt{-1-1-1-1,0-100,1000,1010}
\\
\hline
$18$ & \texttt{-1-1-1-1,-1000,0100,0101}
&
$36$ & \texttt{-1-1-1-1,00-10,0001,0101}
\\
\hline
\end{tabular}
\end{table}

\begin{remark}
\label{in-general-not-Q-factorial}
From Table~\ref{max-cones-fan-R-Delta2-2} we can see that $\mathcal{R}(\Delta_2,2)$ has $24$ simplicial maximal cones and the remaining $12$ are non-simplicial as they have $5$ extremal rays. In particular, $X_{\mathcal{R}(\Delta_2,2)}$ is not $\mathbb{Q}$-factorial.
\end{remark}

\begin{remark}
\label{non-recursive-family-in-general}
As we discussed in \S\,\ref{family-losev-manin-space-revisited}, the universal family over the Losev--Manin moduli space $\overline{\mathrm{LM}}_{m+2}$ is given by the next moduli space $\overline{\mathrm{LM}}_{m+3}$, together with the forgetful morphism $\overline{\mathrm{LM}}_{m+3}\rightarrow\overline{\mathrm{LM}}_{m+2}$ which drops one of the light points. We checked in Proposition~\ref{thm-losev-manin-family-revisited} that the fan $\mathcal{R}(\Delta_1,m)$ is isomorphic to $\mathcal{Q}(\Delta_1,m+1)$ for all $m$. However, it is in general not true that $\mathcal{R}(P,m)$ is isomorphic to $\mathcal{Q}(P,m+1)$.

For example, let $P$ be the $2$-simplex $\Delta_2$ and take $m=2$. We have that $\mathcal{R}(\Delta_2,2)$ and $\mathcal{Q}(\Delta_2,3)$ have a different number of maximal dimensional cones. The number of maximal cones of $\mathcal{R}(\Delta_2,2)$ is $36$ by Table~\ref{max-cones-fan-R-Delta2-2}. On the other hand, by \cite[\S\,3.1]{San05}, we have that $\mathcal{Q}(\Delta_2,3)$ has $108$ maximal dimensional cones. These correspond to the vertices of the fiber polytope $Q(\Delta_2,3)$, which correspond to the polyhedral subdivisions of $3\Delta_2$ induced by the polytope fibration $\Delta_2^3\rightarrow3\Delta_2$ (see \cite{BS92}).
\end{remark}



\end{document}